\documentclass[11pt]{amsart}

\usepackage{amsmath,amsthm,verbatim,amssymb,amsfonts,amscd, graphicx, thm-restate}
\usepackage{graphics}
\usepackage{tikz-cd}
\usepackage{pinlabel}
\usepackage{xcolor}
\usepackage{stackrel}
\usepackage{url}
\usepackage{hyperref}

\usepackage[margin=1in]{geometry}
\newtheorem{theorem}{Theorem}[section]
\newtheorem{corollary}[theorem]{Corollary}
\newtheorem{lemma}[theorem]{Lemma}
\newtheorem{proposition}[theorem]{Proposition}

\newtheorem{question}[theorem]{Question} 
 
\theoremstyle{definition}
\newtheorem{definition}[theorem]{Definition}
\newtheorem{remark}[theorem]{Remark}

\theoremstyle{definition}
\newtheorem{example}{Example}[section]

\newcommand{\leftrarrows}{\mathrel{\raise.75ex\hbox{\oalign{%
  $\scriptstyle\leftarrow$\cr
  \vrule width0pt height.5ex$\hfil\scriptstyle\relbar$\cr}}}}
\newcommand{\lrightarrows}{\mathrel{\raise.75ex\hbox{\oalign{%
  $\scriptstyle\relbar$\hfil\cr
  $\scriptstyle\vrule width0pt height.5ex\smash\rightarrow$\cr}}}}
\newcommand{\Rrelbar}{\mathrel{\raise.75ex\hbox{\oalign{%
  $\scriptstyle\relbar$\cr
  \vrule width0pt height.5ex$\scriptstyle\relbar$}}}}

\makeatletter
\def\leftrightarrowsfill@{\arrowfill@\leftrarrows\Rrelbar\lrightarrows}
\newcommand{\xleftrightarrows}[2][]{\ext@arrow 3399\leftrightarrowsfill@{#1}{#2}}
\makeatother

\definecolor{violet}{rgb}{.6,.6,0}
\definecolor{green}{rgb}{.0,.8,0}

\newcommand{\R}{\mathbb{R}}
\newcommand{\RP}{\R\mathbb{P}}
\newcommand{\CP}{\mathbb{CP}}

\let\int\relax
\newcommand{\int}{\mathring}

\newcommand{\boundary}{\partial}

\newcommand\simtimes{\mathbin{%
    \stackrel{\sim}{\smash{\times}\rule{0pt}{0.7ex}}%
    }}
\newcommand\cp{\mathbb{CP}^2}
\newcommand\cpbar{\smash{\overline{\mathbb{CP}}}^2}

\title[Multisections of 4-manifolds]{Multisections of 4-manifolds}

\author{Gabriel Islambouli}
\address{Department of Pure Mathematics\\University of Waterloo\\Waterloo, ON N2L 3G1, Canada}
\urladdr{https://uwaterloo.ca/scholar/gislambo}
\email{gabriel.islambouli@uwaterloo.ca}

\author{Patrick Naylor}
\address{Department of Pure Mathematics\\University of Waterloo\\Waterloo, ON N2L 3G1, Canada}
\urladdr{https://patricknaylor.org}
\email{patrick.naylor@uwaterloo.ca}

\thanks{PN is supported by an NSERC CGS-D scholarship.}

\begin{document}

\maketitle

\begin{abstract}
We introduce multisections of smooth, closed 4-manifolds, which generalize trisections to decompositions with more than three pieces. This decomposition describes an arbitrary smooth, closed 4-manifold as a sequence of cut systems on a surface. We show how to carry out many smooth cut and paste operations in terms of these cut systems. In particular, we show how to implement a cork twist, whereby we show that an arbitrary exotic pair of smooth 4-manifolds admit 4-sections differing only by one cut system. By carrying out fiber sums and log transforms, we also show that the elliptic fibrations $E(n)_{p,q}$ all admit genus $3$ multisections, and draw explicit diagrams for these manifolds.
\end{abstract}

\section{Introduction}

A trisection is a decomposition of a 4-manifold into three simple pieces, first introduced by Gay and Kirby \cite{GayKir16}. One of the nice features of a trisection is that it encodes all of the smooth topology of a 4-manifold as three cut systems of curves on a surface. It is therefore natural to attempt to realize the cut and paste operations ubiquitous in 4-manifold topology as operations on these surfaces. There has been notable progress in this direction, with recent work realizing the Gluck twist \cite{GayMei18}, the Price twist \cite{KimMil20}, and knot surgery \cite{AraMoe} in this way. Much of the existing literature uses the theory of relative trisections (see \cite{CasGayPin18} for an introduction), which provides a general framework for trisections of manifolds with boundary.

Despite the aforementioned progress, implementing these operations on a trisection can be quite unwieldy in practice. Moreover, what one might expect to be the most natural operation, cutting and regluing one of the pieces, never changes the manifold \cite{LauPoe72}. The goal of this paper is to relax the definition of a trisection in order to provide a more flexible object that is amenable to cut and paste operations. We introduce a decomposition, called a multisection of a 4-manifold, which is the generalization of a trisection to a decomposition which may have more than three pieces. Like trisections, the entire decomposition is encoded as $n$ cut systems of curves on a surface, where $n$ is the number of sectors. Unlike a trisection, one can cut along subsections and re-glue in order to change the diffeomorphism type of the manifold. In particular, in Section \ref{sec:EllipicFibrations}, we show how to use operations on subsections to produce multisection diagrams for the elliptic fibrations $E(n)_{p,q}$, a rich class of simply connected smooth 4-manifolds exhibiting exotic phenomena. Using a theorem of Fintushel and Stern \cite{FinSte97} characterizing the Seiberg-Witten basic classes of these manifolds, we obtain the following corollary.

\begin{restatable*}{corollary}{infManyExoticGenusThree}
\label{cor:infManyExoticGenus3}
There are infinitely many homeomorphism classes of manifolds admitting genus 3 multisections, each of which has infinitely many distinct smooth structures also admitting genus 3 multisections.
\end{restatable*}

More generally, we show that the subtle difference between diffeomorphism and homeomorphism in dimension four is highly compatible with the structure of a 4-section. By work of Curtis, Freedman, Hsiang, and Stong \cite{CurFreHsiSto96} and, independently, Matveyev \cite{Mat96}, any two smooth, homeomorphic, simply connected, closed 4-manifolds are related by a cork twist, i.e., cutting out a contractible compact 4-dimensional submanifold and reguluing it by an involution to produce a new 4-manifold. Interpreting this in the language of multisections, we obtain the following. 

\begin{restatable*}{theorem}{corkCurves}
\label{thm:corkCurves}
Suppose that $X$ and $X'$ are smooth, closed, oriented simply connected 4-manifolds that are homeomorphic, but not diffeomorphic. Then, there exists a surface $\Sigma$, and cut systems $C_1$, $C_2$, $C_3$, $C_4$, and $C_4'$, such that:
\begin{enumerate}
    \item $(\Sigma; C_1, C_2, C_3, C_4)$ is a 4-section diagram for $X$;
    \item $(\Sigma; C_1, C_2, C_3, C_4')$ is a 4-section diagram for $X'$;
    \item There exists a map $\tau: \Sigma \to \Sigma$ such that, $\tau(C_1) = C_1$, $\tau(C_3) = C_3$, and $\tau(C_4) = C_4'$ where $\tau$ is the restriction of a cork twist to $\Sigma$.
\end{enumerate}
\end{restatable*}

\noindent We explicitly realize the change in cut systems required to accomplish the Mazur cork twist in Figure \ref{fig:mazManCorkTwist}.

A natural measure of complexity of a 4-manifold $X$ which arises from this set up is its \emph{multisection genus}, i.e., the minimal genus, $g$, such that $X$ admits a genus $g$ multisection. While this is bounded below by the rank of the fundamental group, the invariant seems to be much more subtle for simply connected 4-manifolds. The standard simply connected 4-manifolds $\#^i S^2 \times S^2 \#^j \cp \#^k \cpbar$ have unbounded trisection genus, but, by contrast, we show the following proposition.

\begin{restatable*}{proposition}{standardManifoldsAreGenusOne}
\label{thm:standardManifoldsAreGenusOne}
The 4-manifolds $\#^i S^2 \times S^2 \#^j \cp \#^k \cpbar$ admit genus one multisections. In particular, these manifolds admit a $(2+2i+j+k)$-section of genus one.
\end{restatable*}

\section{Definitions and existence proofs}

Throughout this paper, we will decompose manifolds using handle decompositions with handles of prescribed indices. We will call an orientable manifold built with handles of index at most $k$ a $k$-handlebody, so that, for example, a $2$-handlebody may contain 1-handles. For notational convenience, we will take $\#^0 S^1 \times S^2 = S^3$ and $\natural^{0} S^1 \times B^3 = B^4$. Since 1-handlebodies are topologically quite simple, decomposing 3- and 4-dimensional manifolds into 1-handlebodies records topological complexity as the complexity of some associated gluing maps. With this in mind, we introduce the main object of study. The reader may refer to Figure \ref{fig:4SectionSchematic} for a visual depiction of this definition.

\begin{definition}
Let $X$ be a smooth, orientable, closed, connected 4-manifold. An $n$\emph{-section}, or \emph{multisection} of $X$ is a decomposition $X = X_1 \cup X_2 \cup \cdots \cup X_n$ such that:
\begin{enumerate}
    \item $X_i \cong \natural^{k_i} S^1 \times B^3$;
    \item $X_1 \cap X_2 \cap \dots \cap X_n = \Sigma_g$, a closed orientable surface of genus $g$;
    \item $X_i\cap X_j=H_{i,j}$ is a 3-dimensional 1-handlebody if $|i-j|=1$, and $X_i \cap X_j = \Sigma_g$ if $|i-j|>1$;
    \item $\partial X_i \cong \#^{k_i} S^1 \times S^2$ has a Heegaard splitting given by $H_{(i-1),i} \cup_{\Sigma} H_{i,(i+1)}$.
\end{enumerate}
\end{definition}

\begin{figure}
    \centering
    \includegraphics[width=0.16\textwidth]{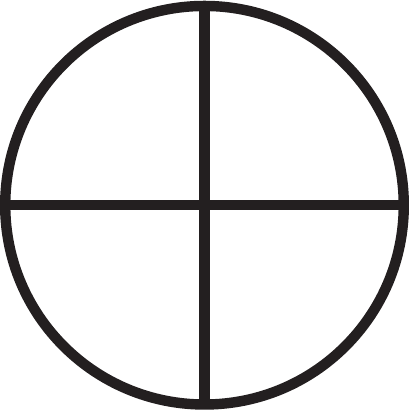}
    \put(-60,48){$X_1$}
    \put(-60,20){$X_2$}
    \put(-28,20){$X_3$}
    \put(-28,48){$X_4$}
    \put(5,35){$H_{34}$}
    \put(-98,35){$H_{12}$}
    \put(-45,-13){$H_{23}$}
    \put(-45,81){$H_{41}$}
    \caption{A schematic picture of a 4-section. The $X_i$ are 4-dimensional 1-handlebodies, and the $H_{ij}$ are 3-dimensional 1-handlebodies. The center of the figure represents a genus $g$ surface, where all of these pieces meet}
    \label{fig:4SectionSchematic}
\end{figure}
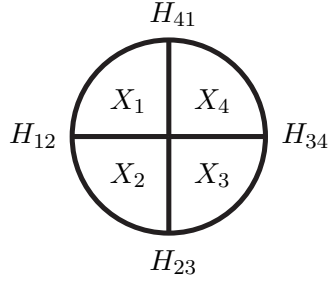

With parameters as in the definition, we will describe this decomposition as a $(g;k_1,\dots,k_n)$ $n$-section of $X$. We will refer to each $X_i$ as a \emph{sector} and $\Sigma_g$ as the \emph{central surface}. We call any pair of 3-dimensional handlebodies which are not the boundary of a common sector a \emph{cross-section}, and we note that any such pair describes a separating, embedded 3-manifold in $X$. Finally, we call the union of the successive sectors $X_{i} \cup X_{i+1} \cup\cdots\cup X_{i+l}$ a \emph{subsection}, and we will denote it by $X_{i,i+l}$. We will call a genus $g$ multisection of a 4-manifold $X$ \emph{thin} if $k_i=g-1$ for all $i$. 

Since every smooth, orientable, closed, connected manifold admits a trisection \cite{GayKir16}, all such 4-manifolds admit multisections. We note that, alternatively, one can prove that all such manifolds admit 4-sections using the gluing techniques of Section \ref{sec:SubsectionOperations}, but we leave the details of the proof to the interested reader. We now define diagrams representing a multisection.

\begin{definition}
A \emph{multisection diagram} is an ordered collection $(\Sigma;\mathcal{C}_1,\dots,\mathcal{C}_n)$ where $\Sigma$ is a surface, $\mathcal{C}_1,\dots,\mathcal{C}_n$ are cut systems for $\Sigma,$ and each triple $(\Sigma; \mathcal{C}_i, \mathcal{C}_{i+1})$ is a Heegaard diagram for $\#^{k_i} S^1 \times S^2$ for some non-negative integer $k_i$ (where the indices are taken mod $n$).
\end{definition}

A multisection diagram gives rise to a multisection via the process illustrated in Figure \ref{fig:fillInTrisection}. Beginning with $\Sigma\times D^2$, one attaches $n$ sets of thickened 3-dimensional 2-handles and a thickened 3-handle along the boundary of $\Sigma\times D^2$ as prescribed by $C_1,\dots,C_n$. By construction, the resulting 4-manifold has $n$ boundary components, each diffeomorphic to $\#^{k_i}S^1\times S^2$. By a theorem of Laudenbach and Po\'enaru \cite{LauPoe72} these boundary components can be uniquely capped off with 4-dimensional 1-handlebodies to produce a smooth closed 4-manifold equipped with a natural $n$-section. Conversely, given an $n$-section of $X$, cut systems determined by the compressing disks for the 3-dimensional handlebodies, $H_{i,j}$, describe an $n$-section diagram for $X$. 

\begin{figure}
    \centering
    \includegraphics[scale = .3]{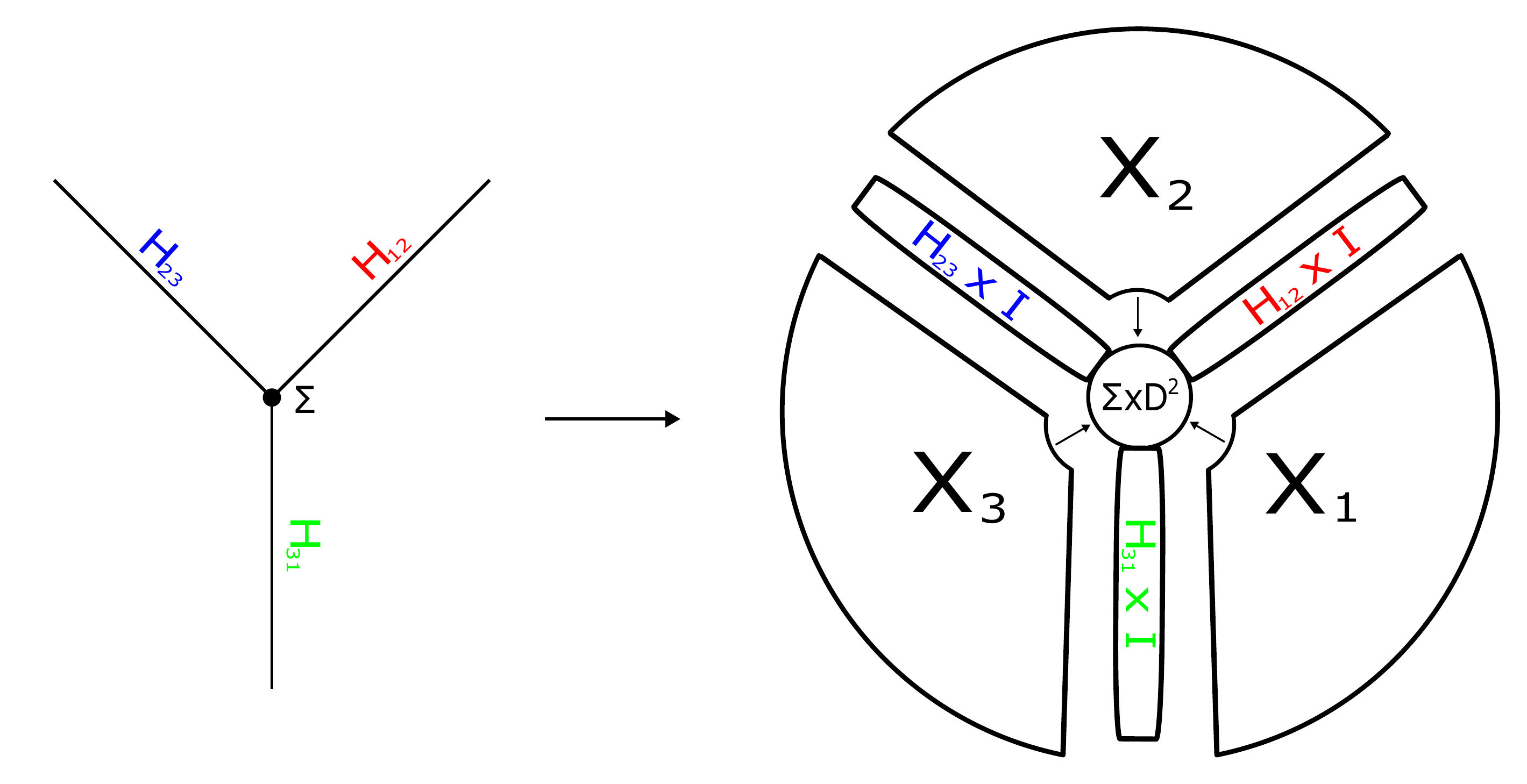}
    \caption{Thickening the 3-manifold obtained by attaching 3-dimensional 2-handles and a 3=handles to the central surface as prescribed by a multisection diagram gives a 4-dimensional manifold which can be uniquely capped off by 4-dimensional 1-handlebodies.}
    \label{fig:fillInTrisection}
\end{figure}

We now give definitions for the corresponding decompositions for manifolds with boundary; this will be the object obtained by removing a number of sectors from a closed multisection. Here, the induced structure on the boundary will be a Heegaard splitting. While such an object can have many sectors, the techniques in Section \ref{sec:handleDecomps} show that one can  reduce the decomposition to two sectors, so we will focus on that case.

\begin{definition}
Let $X$ be a smooth, orientable, connected 4-manifold $X$ with connected boundary. A \emph{bisection} of $X$ is a decomposition $X=X_1 \cup X_2$ such that:
\begin{enumerate}
    \item $X_i \cong \natural^{k_i} S^1 \times B^3$;
    \item $\partial X_1 = H_1 \cup_\Sigma H_2$, where each $H_i$ is a genus $g$ handlebody and $\partial H_i=\Sigma$;
    \item $\partial X_2 = H_2 \cup_\Sigma H_3$, where $H_3$ is a genus $g$ handlebody with $\partial H_3=\Sigma$;
    \item $X_1 \cap X_2 = H_2$;
    \item $\partial X= H_1 \cup_\Sigma H_3$, i.e., has a Heegaard splitting given by $\Sigma$.
\end{enumerate} 
\end{definition}

Bisections have been used previously in the literature. These decompositions were explored by Scharlemann \cite{Sch08}, who proved that any homology 4-ball with boundary $S^3$ admitting a bisection of genus at most three is in fact diffeomorphic to $B^4$. Bisections have also been used by Birman and Craggs \cite{BirCra78} to study invariants of homology 3-spheres, as well as by Ozsv\'{a}th and Szab\'{o} to define cobordism maps in Heegaard-Floer homology \cite{OzsSza06}. We will take a more constructive approach to these objects, and a central object to such a viewpoint is the corresponding diagrammatic object.

\begin{definition}
A \emph{bisection diagram} is an ordered quadruple $(\Sigma,C_1, C_2, C_3)$ where $C_1, C_2$ and $C_3$ are cut systems satisfying $(\Sigma; C_i, C_{i+1})$ is a Heegaard diagram for $\#^{k_i} S^1 \times S^2$ for some non-negative integer $k_i$.
\end{definition}

Note that in the above definition, there is no constraint on the pair of cut systems $C_1$ and $C_3$ and these cut systems form a Heegaard diagram for the boundary of the resulting manifold. As in the closed case, a bisection will determine a bisection diagram and vice versa. We now give proofs of the existence of bisections of a class of smooth, compact 4-manifolds. The proof is similar to that of \cite[Lemma 14]{GayKir16}.

\begin{theorem}
\label{thm:2HandlebodyHasBisection}
Every smooth, compact, connected 2-handlebody with connected boundary admits a bisection.
\end{theorem}

\begin{proof}

Let $X$ be a 4-manifold satisfying the hypotheses of the theorem. Take a handle decomposition of $X$ with a single $0$-handle, $k$ $1$-handles, and $n$ $2$-handles attached along the framed link $L$, and let $X_1$ be the union of the $0$- and $1$-handles. The framed attaching link $L$ lies in $\partial X_1 = \#^kS^1 \times S^2$. Using a tunnel system for $L$ in $X_1$, we may arrange $L$ to lie on a Heegaard surface $\Sigma$ for $\partial X_1$, which decomposes $\boundary X_1$ into two handlebodies $H_1$ and $H_2$. Moreover, when we push $L$ into $H_2$, each component $L_i$ of $L$ has a properly embedded dual disk $D_i\subset H_2$, such that $|L_i \cap D_j| = \delta_{i,j}$. This condition ensures that the result of pushing $L$ into $H_2$ and performing surgery on $L$ is another 3-dimensional handlebody, $H_3$. Moreover, the cobordism between $H_2$ and $H_3$ induced by the 2-handle attachment along $L$ is a 4-dimensional 1-handlebody, which we declare to be $X_2$. The decomposition $X = X_1 \cup X_2$ is the desired bisection. 
\end{proof}

It follows from work of Eliashberg \cite{Eli90} that every compact Stein manifold admits the structure of a 2-handlebody with framing conditions on its attaching links. We therefore immediately obtain the following corollary.

\begin{corollary}
\label{cor:steinMfldsHaveBisections}
Every compact Stein 4-manifold admits a bisection.
\end{corollary}

Recently, Lambert-Cole, Meier, and Starkson \cite{LamMeiSta20} have shown that symplectic manifolds admit Weinstein trisections. Roughly, they do this by realizing a symplectic 4-manifold as a branched covering over $\cp$ and pulling back the symplectic form. It is likely that Stein manifolds admit a similar compatibility with bisections, as they can be realized as branched covers over $B^4$ \cite{LoiPie01}.

\section{Handle Decompositions and Multisections}\label{sec:handleDecomps}

In this section, we describe how to pass between handle decompositions and multisections. The results of these sections will be frequently used in both directions. Passing from a handle decomposition to a multisection diagram will allow us to produce multisections of manifolds with well known handle decompositions, and these multisections can then be modified to produce both new multisections of the same manifold, or multisections of different manifolds. Passing from multisections to handle decompositions will allow us to identify a manifold from its multisection diagram. We will explicitly treat the case of closed multisections, and make notes of the modifications needed for bisections.

\subsection{Handle decompositions}

In this subsection, we show how to pass from a multisection diagram to a Kirby diagram. We will first produce a handle decomposition from a multisection, and then give a Kirby diagram inducing the same handle decomposition. Such a Kirby diagram will have the 2-handles lying in a nice position with respect to a Heegaard surface for the boundary of the 0- and 1-handles, which makes the next well known lemma pertinent. 

\begin{lemma}
\label{lem:dualhandlebody}
Let $H$ be a handlebody and and suppose that $\gamma\subset \partial H$ is a curve such that $|\gamma\cap D|=1$ for some properly embedded disk $D\subset H$. Then, the result of pushing $\gamma$ into $H$, and doing surgery on $\gamma$ is again a handlebody. Moreover, if we do surgery on $\gamma$ using the surface framing, then $\gamma$ bounds a disk in the surgered handlebody.
\end{lemma}

We are now ready to show how to pass between a multisection and a handle decomposition.

\begin{proposition}
\label{prop:handledecomposition}
Suppose that a 4-manifold $X$ admits an $n$-section $X= X_1 \cup X_2 \cup \cdots \cup X_n$ of genus $g$. Then, there is a handle decomposition for $X$ satisfying the following properties:
\begin{enumerate}
    \item The union of the 0- and 1-handles is a collar neighbourhood $H_{1,n}\times [0,1]$ of $H_{1,n}$ in $X_1$, where we identify $H_{1,n}\times \{0\}$ with $H_{1,n}\subset X_1$.
    \item The 2-handles for $X$ are attached sequentially along curves in a neighbourhood of $\partial(H_{1,n}\times \{1\})$. 
\end{enumerate}
\end{proposition}

\begin{proof}
The plan will be to reconstruct $X$ from its multisection, starting with $\Sigma_g\times D^2$. We will parameterize $D^2$ as the unit disk in $\mathbb{C}$. Note that after attaching the thickened handlebody, $H_{1,n} \times I$, to $\Sigma_g\times D^2$, the resulting manifold retracts onto $H_{1,n} \times I$. These are the 1-handles of the handle decomposition, and we will built the rest of the manifold without using any more 1-handles. We parameterize $H_{1,n} \times I$ as $H_{1,n}\times [0,1]$, identifying $H_{1,n}\times\{0\}$ with $H_{1,n}\subset X$, and declare that $H_{1,n}\times \{1\}\subset X_1$. We will use $H[a]$ to denote $H_{1,n}\times [a]\subset X$.

If we view $X- \nu(H_{1,n})$ as a cobordism from $H[1]$ to $H[0]$ then this cobordism is a composition of each sector, $X_i$, where we think of each sector as a cobordism (of manifolds with boundary) from $H_{(i-1),i}$ to $H_{i,(i+1)}$. By assumption, the union of the handlebodies $H_{(i-1),i}$ and $H_{i,(i+1)}$ is $\#^{k_i}S^1\times S^2$. This manifold admits a standard Heegaard splitting with Heegaard surface $\Sigma_g$.  Let $L_i$ be some choice of $g-k_i$ curves bounding disks in $H_{i,(i+1)}$ which are dual to curves bounding disks in  $H_{(i-1),i}$. Considering this as a link with surface framing, we obtain a $g-k_i$ component link $L$ satisfying the assumptions of Lemma \ref{lem:dualhandlebody} with respect to $H_{(i-1),i}$. Attaching 2-handles along this link produces a cobordism from $H_{(i-1),i}$ to $H_{i,(i+1)}$.

In fact, this cobordism is diffeomorphic to $\natural^{k_i}S^1\times B^3$ since the 2-handles cancel the $g-k_i$ of the 1-handles of $H_{(i-1),i}\times I$. By \cite{LauPoe72}, any two ways of attaching $\natural^{k_i}S^1\times B^3$ are equivalent, so we may take the sector $X_i$ to be this specific cobordism. In other words, we have now exhibited $X- \nu(H_{1,n})$ as a union of $g$ 1-handles and $(g-k_1)+\cdots+(g-k_{n})$ 2-handles, attached as required. What remains is a genus $g$ 1-handlebody, and this forms the 3-handles and 4-handle for $X$.

We next arrange the attaching circles for all of the 2-handles to lie in $H[1]$. A schematic of this process is illustrated in Figure \ref{fig:MStoKD}. First, decompose the boundary of $X- (\nu(H_{1,n})\cup \Sigma\times D^2)$  as $H[1] \cup (\Sigma\times [0,1]) \cup H[0]$. Since each 2-handle attaching curve in the above construction may be isotoped into $\Sigma_g$, we can push the 2-handles into various levels of the $H_{i,i+1}$. In particular, pushing each 2-handle sufficiently deep into the handlebodies ensures that there exists a collar neighbourhood of $\Sigma \times [0,1]$ which we parameterize as $\Sigma \times [0,1] \times [0,1]$.

We slightly modify the attaching region of each of these 2-handles. In particular, we push the 2-handle attaching curves between $H_{i-1,i}$ and $H_{i,i+1}$ into $\Sigma \times {\{\frac{i}{n}\} \times \{\frac{n-i}{n}}\}$. Using this parameterization, each 2-handle is attached at a shallower collar of $\Sigma \times [0,1]$. Thus the second factor of the product structure on $\Sigma \times [0,1] \times [0,1]$ can be used to transport each of these 2-handle attaching curves into  $\Sigma \times [0,1] \times \{0\} \subset H[1]$.
\end{proof}

From the above proof we see that a thin multisection has a particularly nice handle structure, where each sector corresponds to a single 2-handle attachment. We immediately obtain the following corollary.

\begin{corollary}
\label{prop:handleDecompFromThinMultisection}
Suppose that a 4-manifold $X$ admits a thin $n$-section of genus $g$. Then, $X$ admits a handle decomposition with one $0$-handle, $g$ 1-handles, $n$ 2-handles, $g$ 3-handles, and one 4-handle.
\end{corollary}

\begin{figure}
    \centering
    \includegraphics[scale=.4]{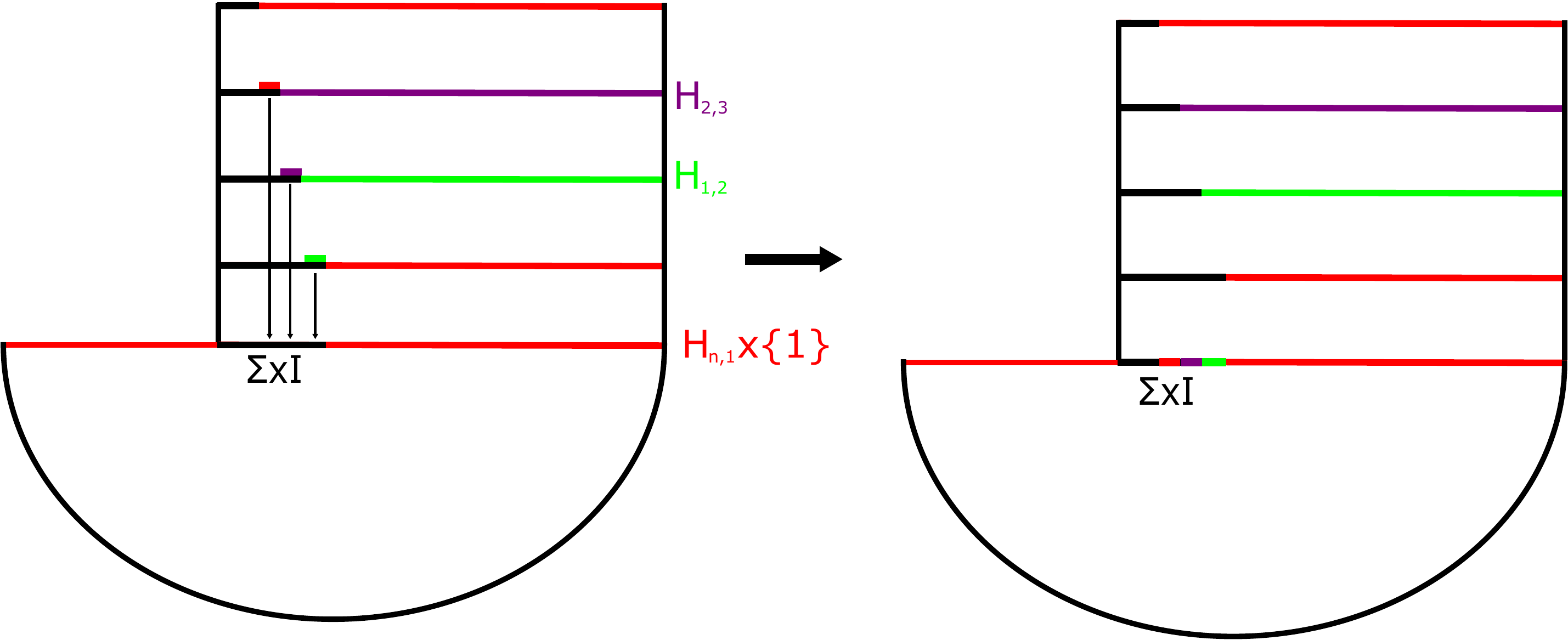}
    \caption{A schematic illustrating the process in Proposition \ref{prop:kirbydiagramfrommultisection}. The short colored curves indicate the 2-handle attachments giving rise to the subsequent 3-dimensional handlebodies. These 2-handles can be projected down to a collar neighbourhood of the multisection surface by a canonical product structure on the cobordism between the handlebodies in this collar neighbourhood.}
    \label{fig:MStoKD}
\end{figure}

\subsection{Drawing Kirby diagrams}

In this subsection, we further refine Proposition \ref{prop:handledecomposition} to produce Kirby diagrams from a given multisection diagram. We remind the reader that a Kirby diagram is drawn in a particular representation of some connected sum of copies of $S^1 \times S^2$. Thus, our task is to take the handle decomposition above and concretely draw the attaching link for each 2-handle. 

\begin{proposition}
\label{prop:kirbydiagramfrommultisection}
Suppose that a 4-manifold $X$ admits an $n$-section of genus $g$, described by a multisection diagram $(\Sigma,\mathcal{C}_I)$. Then, a Kirby diagram for $X$ may be obtained via the following algorithm:
\begin{enumerate} 
 \item Standardize one set of curves, say $\mathcal{C}_1$, to look like the cut system on the left of  Figure \ref{fig:cutSystemToSpheres}.
  \item Cut along small neighbourhoods of the curves in $\mathcal{C}_1$ and place two spheres on the two boundary components resulting from each cut, as illustrated on the right of Figure \ref{fig:cutSystemToSpheres}. Each pair of spheres is the attaching region for a 1-handle in the Kirby diagram.
 \item For each $1<k\leq n$, draw the curves of $\mathcal{C}_k$ that are dual to a curve in $\mathcal{C}_{k-1}$, ``pushed in'' to the surface. That is, view the surface in the diagram as $\Sigma\times \{0\}\subset \Sigma\times [0,1]$ with the $[0,1]$ factor pointing away from the point at infinity, and attach the relevant curves of $\mathcal{C}_k$ at level $\Sigma\times \{\frac{n-k}{n}\}$.
 \item Frame each curve with its corresponding surface framing. 
\end{enumerate}
\end{proposition}

\begin{proof}
Given a multisection of $X$, begin with a handle decomposition as in Proposition \ref{prop:handledecomposition}. Standardising $\mathcal{C}_1$ to appear as in Figure \ref{fig:cutSystemToSpheres} and replacing each curve with a pair of spheres produces a Kirby diagram of the 1 handles of this handle decomposition, together with the Heegaard surface of the boundary of the 1-handles. In Proposition \ref{prop:handledecomposition}, the 2-handles are attached with surface framing in progressively shallower neighbourhoods of a collar neighbourhood $\Sigma \times [0,1]$ of this Heegaard surface, where the $[0,1]$ factor goes towards the ``inside'' handlebody (i.e. the handlebody not containing the point at infinity). Projecting them down into the boundary of the 0- and the 1-handles as in Figure \ref{fig:MStoKD} gives the desired result.
\end{proof}

\begin{figure}
    \centering
    \includegraphics[scale = .18]{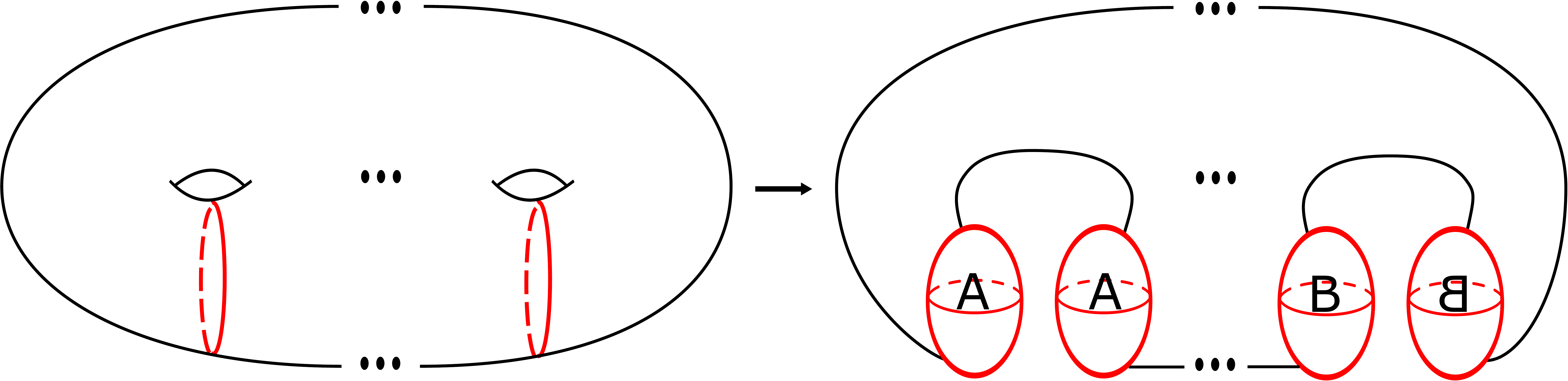}
    \caption{Each curve in the standard cut system becomes a pair of 2-spheres when converting to a Kirby diagram.}
    \label{fig:cutSystemToSpheres}
\end{figure}

\begin{remark}
This is a generalization of the way one usually obtains a Kirby diagram from a trisection diagram. The main difference is that in this case, one must take care to attach the 2-handle curves in the correct order. 
\end{remark}

\begin{example}
\label{ex:kirby2CP2}
We illustrate these methods with an example, finding a Kirby diagram from the multisection diagram in Figure \ref{kirby2CP2} (a), where the order of the curves is red, blue, green, and orange. The red curve is standard, so we can obtain a Kirby diagram by pushing the curves of the diagram out of the surface as in Proposition \ref{prop:kirbydiagramfrommultisection}. The first three curves are given their surface framing, and the red curve becomes a 1-handle (in dotted circle notation). An easy Kirby calculus argument shows that this manifold is indeed $\cp\#\cp$. 

\begin{figure}[ht]
    \centering
    \includegraphics[width=0.9\textwidth]{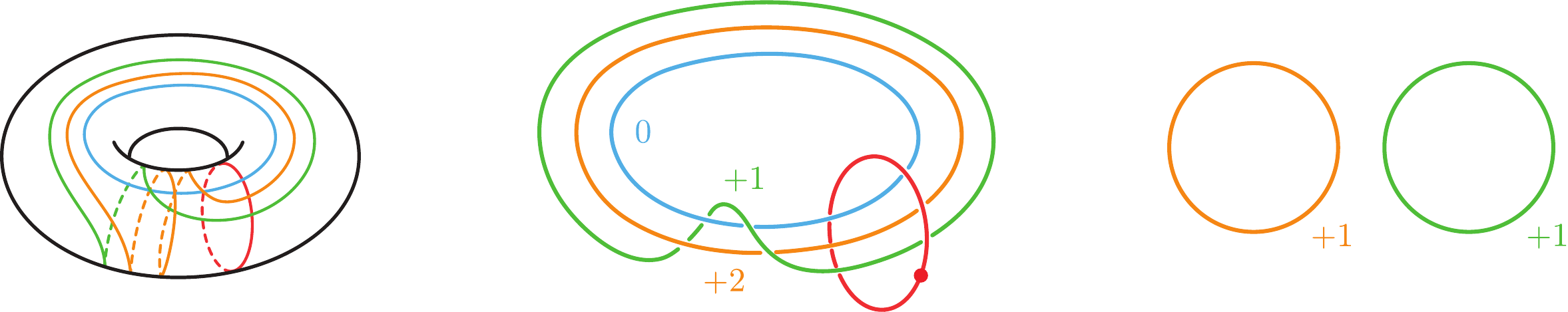}
    \put(-377,-20){(a)}
    \put(-222,-20){(b)}
    \put(-60,-20){(c)}
    \caption{Obtaining a Kirby diagram from a multisection. In (a), we start with a genus 1 multisection diagram. After pushing the curves out correctly, we obtain the Kirby diagram in (b). After slides and cancelling a 1/2- handle pair, we obtain the diagram of $\cp\#\cp$ in (c).}%
    \label{kirby2CP2}
\end{figure}
\end{example}

\begin{remark}
These arguments apply equally well to the case of bisections. In this case, we can start our handle decomposition at either $H_1$ or $H_2$, and we will have no 3- or 4-handles. As a consequence, the procedure for converting a bisection diagram to a Kirby diagram is the same as in the closed case, except that there are no 3- or 4-handles.
\end{remark}

\subsection{Multisections from handle decompositions}\label{sec:handlesToMultisections}

We now show how to convert handle decompositions and Kirby diagrams into multisections and multisection diagrams. For trisections, this process was outlined in Lemma 14 of \cite{GayKir16}, and was studied carefully in \cite{MeiSchZup16}, where the authors introduced the notion of a Heegaard-Kirby diagram. We follow these references but make appropriate modifications for our setting. We also take extra care to obtain a multisection diagram where possible, rather than just an abstract multisection. The main advantages of these modifications are extra flexibility, and the reduction of the genus of the surface. 

\begin{definition}
Let $H$ be a handlebody, and let $c_1$ be a curve on $\boundary H$. We say that $c_1$ is \emph{dual} to $H$ if there exists a properly embedded disk $D$ in $H$ such that $|c_1 \cap D| = 1$. We denote by $H(c_1)$ the handlebody obtained by pushing $c_1$ into $H$ and doing surgery along $c_1$ with surface framing.
\end{definition}

Note that Lemma \ref{lem:dualhandlebody} guarantees that the manifold $H(c_1)$ is indeed a handlebody. Given a curve $c_2$ dual to $H(c_1)$ we will write $H(c_1, c_2)$ to denote the handlebody $H(c_1)(c_2)$. We inductively define $H(c_1,\dots.,c_n)$ in the obvious manner. We emphasize that this notation specifies an ordered process; for example, the curve $c_2$ may not be dual to $H$. Using this notation, we define an intermediary object between a Kirby diagram and a multisection diagram.

\begin{definition}
A \emph{multisection prediagram} is a triple $(\Sigma_g, C_0, (c_1,\dots,c_n))$ where:
\begin{enumerate}
    \item $\Sigma_g$ is a genus $g$ surface, $C_0$ is a cut system determining a handlebody $H$, and $(c_1, \dots, c_n)$ is an ordered collection of curves on $\Sigma$
    \item $c_1$ is dual to $H$, and for all $i>1$, $c_i$ is dual to $H(c_1,\dots,c_{i-1})$.
    \item $H(c_1,\dots, c_n) = H$.
\end{enumerate}
\end{definition}

To construct a multisection diagram we will often start by finding a multisection prediagram within a Kirby diagram. The procedure for doing this is deferred briefly until Lemma \ref{lem:KDtoPrediagram}. Our next task is to show how to convert this prediagram into an honest multisection diagram, which is the content of the following lemma. 

\begin{lemma}
A multisection prediagram determines a thin multisection diagram. 
\end{lemma}

\begin{proof}
Given a multisection prediagram, $(\Sigma, C_0, (c_1, \dots, c_n))$ we obtain a multisection diagram \allowbreak $(\Sigma;C_0, \allowbreak C_1,\dots,C_{n-1})$  via the following inductive procedure. Suppose we have constructed the cut systems $C_0,\dots,C_i$; we will show how to construct $C_{i+1}$. Let $c_{i'}^1$ be a curve bounding a disk in $H(c_1,\dots,c_i)$ which is dual to $c_{i+1}$, and let $C_i' = (c_{i'}^1, c_{i'}^2,\dots, c_{i'}^g)$ be a cut system of $H(c_1,\dots,c_i)$ containing this curve. For all $k \neq 1$ we may remove any intersections between $c_{i'}^k$ and $c_{i+1}$ by sliding $c_{i'}^k$ over $c_{i'}^1$. Denote the result of sliding $c_{i'}^k$ by $c_{i+1}^k$. The cut system $C_{i+1}$ is given by $(c_{i+1}, c_{i+1}^2,\dots,c_{i+1}^n)$. By construction, each adjacent pair of cut systems in this diagram are a Heegaard splitting of $\natural^{g-1}S^1\times S^2$.
\end{proof}

We now show how to obtain a multisection prediagram from a Kirby diagram. This process, and its proof, are essentially the inverse of the procedure outlined in Proposition \ref{prop:kirbydiagramfrommultisection}. 

\begin{lemma}\label{lem:KDtoPrediagram}
Let $X$ be a 4-manifold presented as a Kirby diagram drawn in $S^3$. Let $\Sigma$ be a Heegaard surface for the boundary of the 0- and 1-handles, and $C_0$ be a cut system for $H$, one of the handlebodies in this Heegaard decomposition. Let $L = (L_1, L_2, \dots, L_n)$ be the framed attaching link for the 2-handles. Suppose that $L$ is projected onto $\Sigma$ so that:
\begin{enumerate}
    \item Each component $L_i$ is embedded on $\Sigma$ with handle framing induced by the surface framing;
    \item The projection of $L_1$ on $\Sigma$ is dual to $H$, and for $i>1$, $L_i$ is dual to $H(L_1, \dots, L_{i-1})$;
    \item The projection of each $L_i$ respects the original link in the following sense: parameterize a collar of $\Sigma\subset H$ as $\Sigma\times [0,1]$ where $\Sigma$ is identified with $\Sigma \times \{0\}$. Pushing each $L_i$ into the level $\Sigma\times \{(n-i)/n\}$ recovers the original Kirby diagram.
\end{enumerate}
Then, the triple $(\Sigma, C_0, (L_1,\dots, L_n))$ is a multisection prediagram for $X$.
\end{lemma}

\begin{proof}
The fact that this is a multisection prediagram follows directly from the definitions. To see that this does indeed result in a multisection prediagram for the given manifold, one can convert the prediagram to a diagram and use the process of Proposition \ref{prop:kirbydiagramfrommultisection} to see that this returns the Kirby diagram we started with.
\end{proof}

\begin{example} 
We illustrate the method of Lemma \ref{lem:KDtoPrediagram} with two simple examples, obtaining a bisection diagram for $B^4$ and a 4-section diagram for $S^2\times S^2$. For a more involved use of this procedure, see Section \ref{sec:mazurex}. 

To obtain a 4-section diagram of $S^2\times S^2$, we start with the usual Kirby diagram with two 2-handles. As in Lemma \ref{lem:KDtoPrediagram}, we take a Heegaard surface for the boundary of 0- and 1-handles ($S^3$ in this case), so that we can project the 2-handles onto this surface with the property that their framings agree with the induced surface framing. Because there are so few curves, the duality condition is met automatically. In fact, because the genus of this multisection is one, the multisection prediagram shown in Figure \ref{fig:multisectiondiagramS2xS2} is already a multisection diagram. To verify this diagram is correct, we can obtain a Kirby digram for $S^2\times S^2$ from this diagram using Proposition \ref{prop:kirbydiagramfrommultisection}. 

\begin{figure}[ht]
    \centering
    \includegraphics[width=0.9\textwidth]{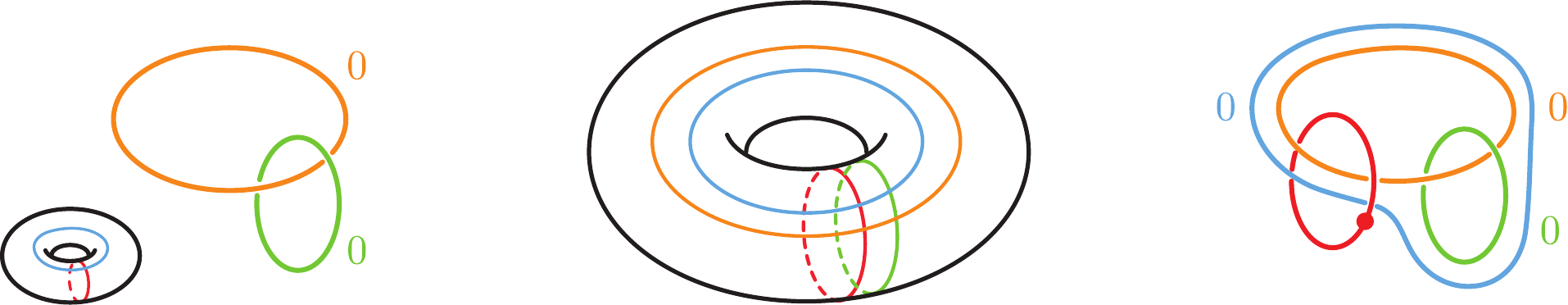}
    \put(-370,-20){(a)}
    \put(-208,-20){(b)}
    \put(-55,-20){(c)}
    \caption{The process of obtaining a 4-section diagram for $S^2\times S^2$. In (a), a Kirby diagram with a Heegaard surface for $S^3$. In (b), the curves projected onto this Heegaard surface. In (c), the Kirby diagram corresponding to this 4-section diagram.}
    \label{fig:multisectiondiagramS2xS2}
\end{figure}

Similarly, we can obtain a bisection diagram for $B^4$. We begin with the Kirby diagram in Figure \ref{fig:multisectiondiagramB4}, and use a genus one Heegaard surface for the boundary of the 0- and 1-handles ($S^1\times S^2$ in this case). We can easily project the 2-handle curve onto this surface, and as before, this bisection prediagram is already a bisection diagram. Note that essentially the same process may be used to produce an unbalanced trisection diagram for $S^4$. 

\begin{figure}[ht]
    \centering
    \includegraphics[width=0.9\textwidth]{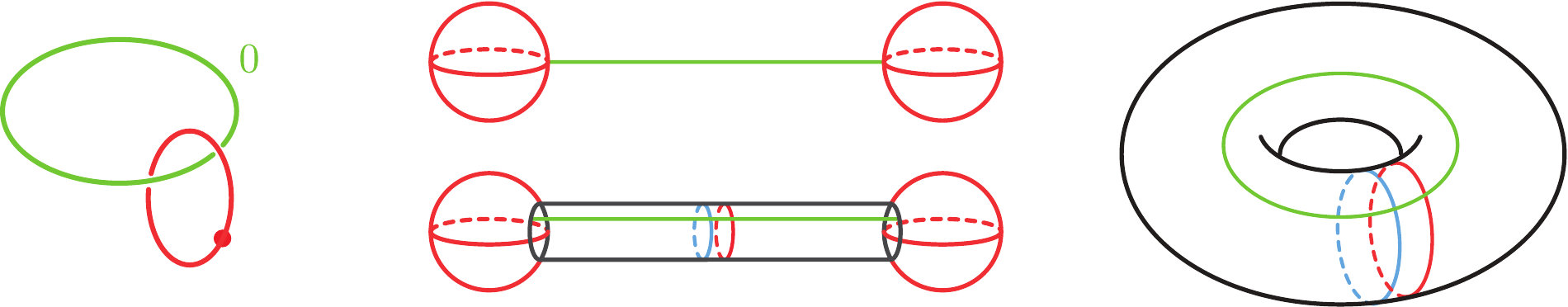}
    \put(-395,-20){(a)}
    \put(-237,-20){(b)}
    \put(-66,-20){(c)}
    \caption{The process of obtaining a bisection diagram from a Kirby diagram of $B^4$. In (a), a Kirby diagram for $B^4$. In (b), a Heegaard surface for the boundary of the 0/1-handles of this decomposition. In (c), the 2-handle curve projected onto this Heegaard surface.}
    \label{fig:multisectiondiagramB4}
\end{figure}
\end{example}

\begin{example}
Some of the simplest 4-manifolds with boundary are knot traces, usually denoted $X_n(K)$. These manifolds are obtained by attaching an $n$-framed 2-handle to $B^4$ along a knot $K$. Adapting a technique often used in Heegaard-Floer homology for drawing doubly pointed Heegaard diagrams of a knot \cite{OszSza04Knot}, we describe a quick way to produce bisection diagrams of $-X_n(K)$ from a knot diagram for $K$. This process is illustrated in Figure \ref{fig:knottrace}.

To produce this multisection, begin by flattening a knot diagram for $K \subset S^3$ into a plane $P$, and thicken the resulting 4-valent graph, $G$, to a handlebody. The boundary of this handlebody is a Heegaard surface, $\Sigma$, and will be our bisection surface. To obtain $C_1$, take loops around each bounded component of $P \backslash G$ and push them onto $\Sigma$. To obtain $C_2$, take a meridian of $K$ together with a curve at each vertex of $G$ which depends on the crossing information (illustrated in Figure \ref{fig:knottracecrossing}). 

The knot $K$ can be projected to $\Sigma$ so that it only intersects $C_2$ in the meridian curve. We may perform twists of $K$ about this curve until the surface framing matches our desired framing, and we call the resulting curve $K^p$. The cut system $C_3$ is obtained from $C_2$ by replacing the meridian curve with $K^p$. This bisection diagram produces a manifold with a number of cancelling handles, together with a 2-handle attached along the knot with surface framing. A minor technical note: the bisection diagram starts with what is usually the negatively oriented handlebody of $S^3$ so that, by orientation conventions, we have actually produced a multisection diagram for $-X_n(K)$.

\begin{figure}
    \vspace{2mm}
    \centering
    \includegraphics[width=0.7\textwidth]{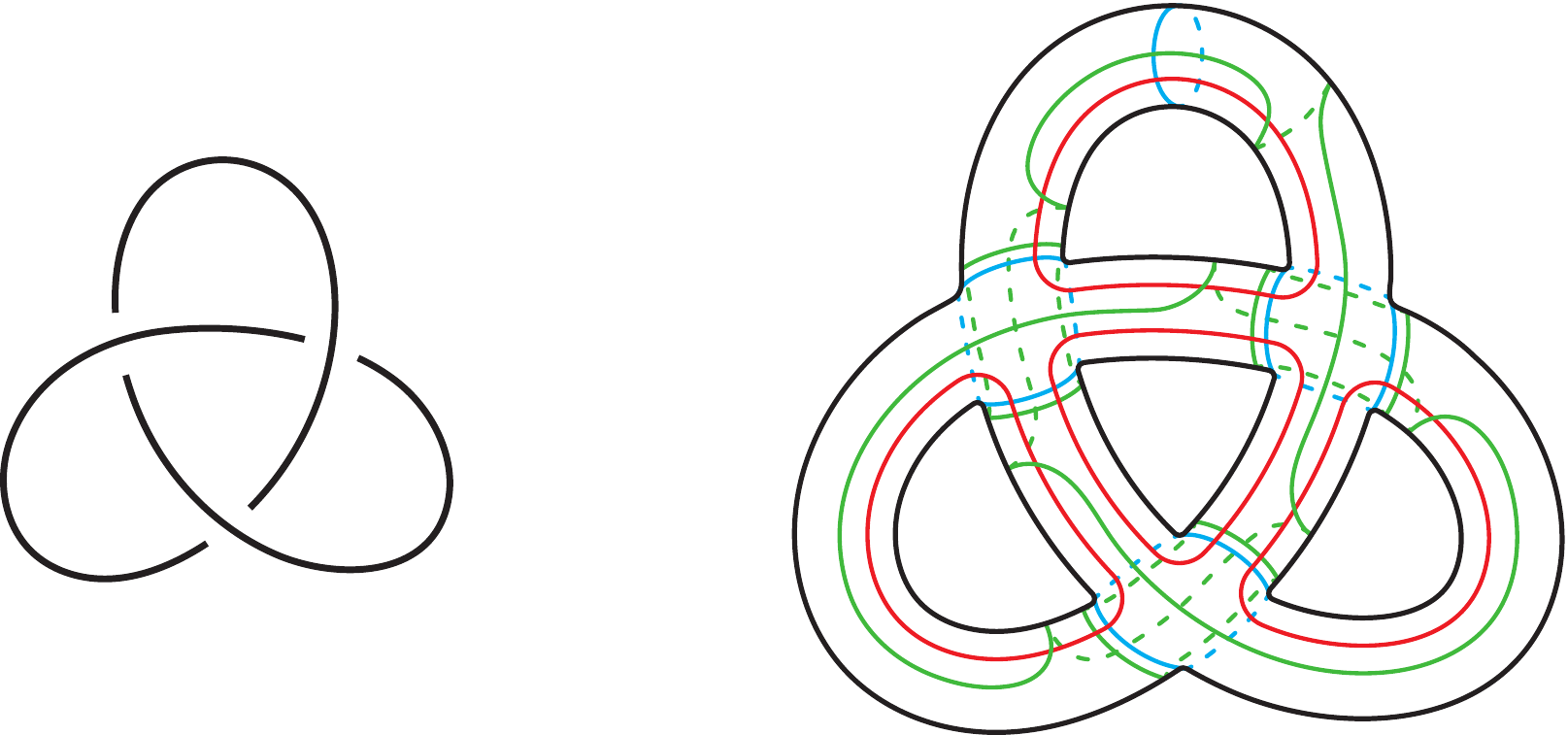}
    \put(-250,105){\Large $+1$}
    \put(-105,-15){$-X_{+1}(3_1)$}
    \caption{A bisection diagram for $-X_{+1}(3_1)$; the boundary of this bisection is the Poincar\'e homology sphere. The cut systems $C_1$, $C_2$, and $C_3$ are red, blue, and green, respectively.}
    \label{fig:knottrace}
\end{figure}

\begin{figure}
    \centering
    \includegraphics[width=0.4\textwidth]{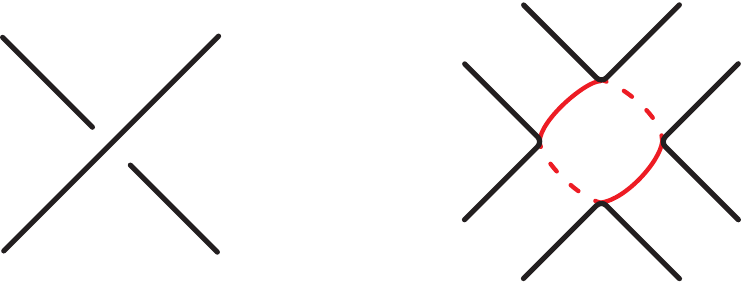}
    \caption{The curve corresponding to a crossing of a knot, which records the 3-dimensional crossing information.}
    \label{fig:knottracecrossing}
\end{figure}

\end{example}

\section{Subsection operations}\label{sec:SubsectionOperations}

In this section, we will show how to cut and reglue along subsections of a multisection, so that the resulting manifold inherits a natural multisection structure. We will also show how the resulting multisection diagrams change under such a procedure. In a trisection, each subsection or its complement is diffeomorphic to $\natural^k S^1 \times B^3$, so cutting and regluing a sector can never produce a different manifold \cite{LauPoe72}. By contrast, using techniques of this section, we will show in Theorem \ref{thm:corkCurves} that any two smooth structures of a fixed 4-manifold are related by modifying some 4-section.

The boundary of a subsection has a natural Heegaard splitting, and so gluing multisections must respect this structure. The following is a formal definition of a map respecting a given Heegaard splitting.

\begin{definition}
The \emph{Goeritz group} of a Heegaard splitting $M=H_1\cup_\Sigma H_2$ is the subgroup of $\text{Mod}(\Sigma)$:
$$\text{Mod}(M;H_1,H_2,\Sigma)=\{\phi \in \text{Mod}(\Sigma): \text{$\phi$ extends across $H_1$ and $H_2$}\}.$$
\end{definition}

Equivalently, this is the subgroup mapping classes of $M$ fixing $\Sigma$ together with a normal orientation of the surface. The structure of this group can be quite intricate: even determining a presentation for Heegaard splitting of the 3-sphere is difficult. Such a presentation is currently unknown for genus greater than three, and the genus three case was resolved only recently \cite{FreSch18}. Moreover, if the Heegaard splitting is stabilized, which is often the case in our setting, then the Goeritz group contains pseudo-Anosov elements \cite{JohRub13}. 

Suppose $X$ is a 4-manifold with multisection $X_I=(X_1,\dots,X_n)$. Consider a subsection $X_{i,j}$, whose boundary is given by the Heegaard splitting $M^3 = H_i \cup_\Sigma H_{j+1}$. Given a map $\phi \in \text{Mod}(M;H_i,\allowbreak H_{j+1},\Sigma)$, the manifold $(X-X_{i,j}) \cup_{\phi} X_{i,j}$ inherits the structure of a multisection. While an arbitrary map of a 3-manifold need not respect a given Heegaard splitting, the following lemma states that a map does respect the splitting up to stabilizations. 

\begin{lemma}\label{lem:MapFixHeegaard}
Let $M$ be a closed 3-manifold, $\phi: M \to M$ be a homeomorphism, and $M= H_1 \cup_{\Sigma} H_2$ be a Heegaard splitting. Then there exists a stabilization of this Heegaard splitting, $H_1' \cup_{\Sigma'} H_2'$, for which $\phi \in \textnormal{Mod}(M;H_1',H_2',\Sigma')$.
\end{lemma}

\begin{proof}
Consider the Heegaard splitting $\phi(H_1) \cup_{\phi(\Sigma)} \phi(H_2)$. This decomposition may not be isotopic to $H_1 \cup H_2$, but the Reidemeister-Singer theorem \cite{Rei33} \cite{Sin33} guarantees that these Heegaard splittings can be made isotopic after some number of stabilizations. Since the stabilization operation commutes with homeomorphisms, $\phi$ fixes some stabilization of $H_1 \cup H_2$.
\end{proof}

Our next goal is to modify a multisection to induce a Heegaard splitting on the boundary, while leaving the underlying manifold unchanged. This operation will be the analogue of the unbalanced stabilization introduced in \cite{MeiSchZup16}. 

\begin{definition}\label{def:Stabilization}
Let $X$ be a multisected 4-manifold with sectors $X_I=X_1,\dots, X_n$, and suppose that this $n$-section is prescribed by a multisection diagram $(\Sigma, C_{1},\dots,C_{n})$, where $C_k$ is a cut system for $H_{k,k+1}$. Let $i$ and $j$ be natural numbers with $i < j$. An \emph{$(i,j)$-stabilization} is the multisection obtained by taking the connected sum of $X$ with the multisection of $S^4$ corresponding to the diagram in Figure $\ref{fig:stabDiagram}$. Equivalently, one introduces additional genus and extends the cut systems by one of either two curves $c_1$ or $c_2$, with $|c_1 \cap c_2|=1$. If $i \leq k \leq j-1$ then we extend $C_{k,k+1}$ by $c_1$, and otherwise, we extend $C_{k,k+1}$ by $c_2$.  
\end{definition}

This operation does not change the underlying manifold, since it amounts to taking a connected sum with $S^4$. To see this explicitly, observe that we can convert the diagram in Figure \ref{fig:stabDiagram} to a Kirby diagram consisting solely of a canceling 1-2 and 2-3 pair. Moreover, if a cross manifold is given by $H_{k,k+1}\cup_\Sigma H_{l,l+1}$ where $k<l$, then an $(i,j)$-stabilization changes the induced Heegaard splitting by a stabilization, if $i<k<j<l$, and by a connected sum with $S^1 \times S^2$ otherwise. Therefore, given any cross-section $H_{k,k+1}\cup_\Sigma H_{l,l+1}$ with $k-l \geq 2$, the $(k+1,l+1)$-stabilization will induce a stabilization on this cross-section. This is the main observation for Proposition \ref{prop:reglueAfterStab}, which essentially states that subsections of a multisection can be reglued while respecting the structure of the multisection.

\begin{figure}[ht]
    \centering
    \includegraphics[width=0.3\textwidth]{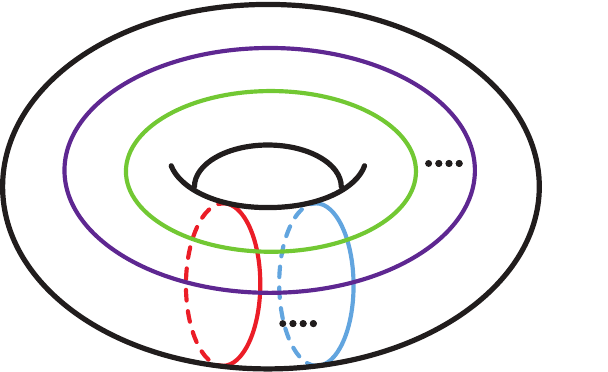}
    \put(-93,10){\small $c_1$}
    \put(-27,47){\small $c_2$}
    \caption{A multisection diagram for $S^4$. Taking the connected sum with this diagram does not change the underlying 4-manifold, but modifies cross-sections by either adding an $S^1 \times S^2$ summand, or inducing a stabilization on their Heegaard splitting.}
    \label{fig:stabDiagram}
\end{figure}

We also define a stabilization operation for bisections. Suppose that $X$ is a 4-manifold with non-empty connected boundary, with bisection $X= X_1 \cup X_2$ prescribed by a bisection diagram $(\Sigma; C_{1}, C_{2}, C_3)$, and that the Heegaard splitting on $\partial X$ is given by the curves $C_1$ and $C_3$. To perform an \emph{$1$-stabilization} of $X_1 \cup X_2$, take an arc $\alpha\subset \partial X_2$, and set $X_1'=X_1\cup \nu(\alpha)$ , and $X_2'=X_2- \nu(\alpha)$. Exchanging the roles of $X_1$ and $X_2$ in the previous discussion gives the process for a 2-stabilization. Diagrammatically, an $i$-stabilization corresponds to taking the connected sum with one of the two diagrams in Figure \ref{fig:bisectionStabilization}. 

\begin{figure}[ht]
    \centering
    \includegraphics[width=0.5\textwidth]{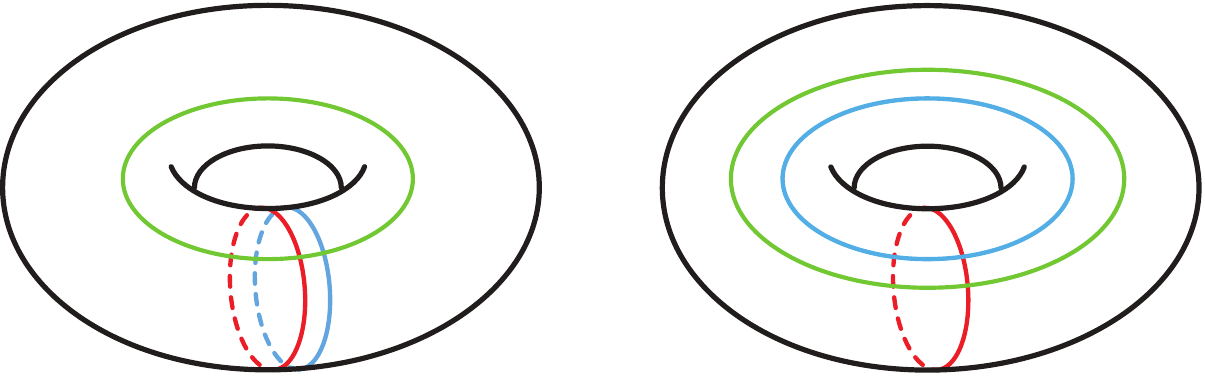}
    \caption{The two possible stabilizations for a bisection (where $C_1$, $C_2$, $C_3$ are colored red, blue, and green respectively). On the left, the diagram corresponding to a $1$-stabilization, and on the right, a $2$-stabilization.}
    \label{fig:bisectionStabilization}
\end{figure}

\begin{proposition}
\label{prop:reglueAfterStab}
Let $X$ be a multisected 4-manifold with sectors $X_I=X_1,\dots, X_n$, and let $X_{i,j}$ be a subsection. Then for any $\phi \in \textnormal{Mod}(\partial X_{i,j}),$ there exists some stabilization, $X_I'$, of $X_I$ such that $(X_i' - X_{i,j}') \cup_{\phi} X_{i,j}'$ has a multisection with sectors $X_I'$.
\end{proposition}

\begin{proof}
If $|i-j| = 1$, then the Heegaard splitting of $\partial X_{i,j}$ is the unique genus $g$ Heegaard splitting of $\#^k S^1 \times S^2$, and so any regluing map respects this Heegaard splitting. If $|i-j| > 1$, a given map may not respect the Heegaard splitting $\partial X_{i,j} = H_i \cup H_{j+1}$, but by Lemma \ref{lem:MapFixHeegaard}, this map respects some stabilization of $H_i \cup H_{j+1}$, which we call  $H_i' \cup H_{j+1}'$. After sufficiently many $(i+1,j+1)$ stabilizations (which induce stabilizations on $H_i \cup H_{j+1}$), we get a new multisection with sectors $X_I'$, and $\partial X_{i,j}' = H_i' \cup H_{j+1}'$. Since $\phi \in \text{Mod}(H_i' \cup H_{j+1}';H_i',H_{j+1}',\Sigma)$, the new manifold $(X' - X_{i,j}') \cup_{\phi} X_{i,j}'$ inherits a natural multisection structure coming from the previous pieces.
\end{proof}

Next, we study the effect of this cut and paste operation on diagrams. The following discussion is illustrated in Figure \ref{fig:gluingCutSystems}. Fix the central surface $\Sigma$ for a given multisection of a 4-manifold $X$, with sectors $X_I=X_1,\dots,X_n$. On the level of the cut systems, the effect of cutting out and regluing $X_{i,j}$ by a map $\phi: \boundary X_{i,j} \to \boundary X_{i,j}$ which respects the Heegaard splitting $\partial X_{i,j}= H_i \cup H_{j+1}$ amounts to applying $\phi\vert_{\Sigma}$ to each of the cut systems $C_i, C_{i+1}, \dots C_{j}, C_{j+1}$. Since $\phi$ extends across $H_{i}$ and $H_{j+1}$, $\phi(C_i)$ and $\phi(C_{j+1})$ are handle slide equivalent to $C_i$ and $C_{j+1}$ respectively. This leads immediately to the following proposition, which will be used frequently throughout this paper. 

\begin{figure}[ht]
    \centering
    \includegraphics[scale=.33]{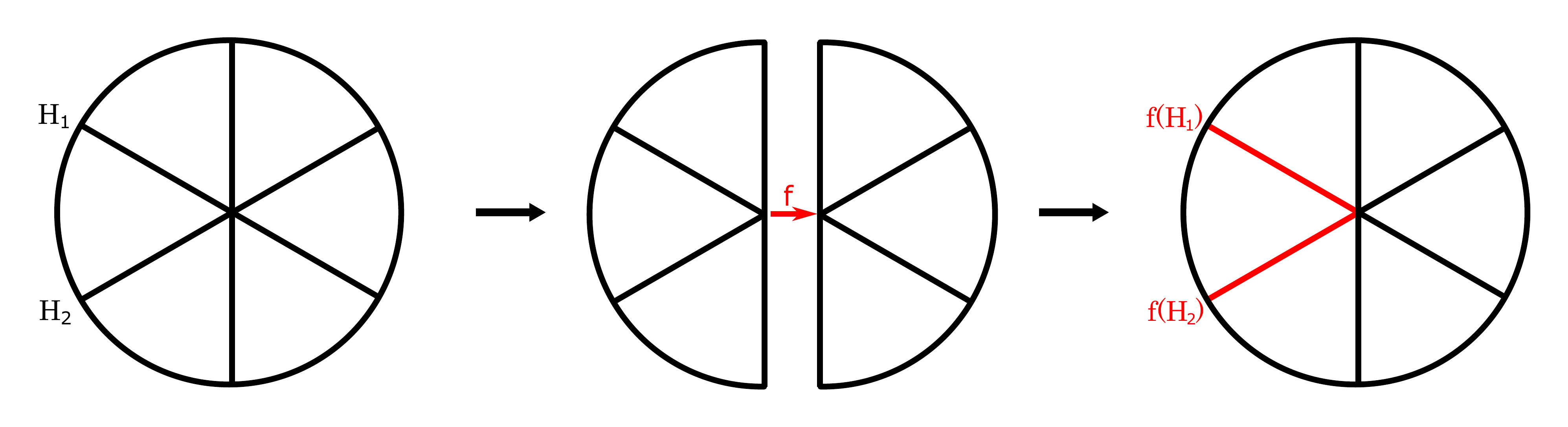}
    \caption{ A schematic illustrating the process of cutting out and regluing a subsection by a map respecting the boundary Heegaard splitting. The handlebodies on the boundary do not change, but the handlebodies in the subsection being reglued may change.}
    \label{fig:gluingCutSystems}
\end{figure}

\begin{proposition}
\label{prop:CurvesAfterGluing}
Let $X$ be a multisected 4-manifold with sectors $X_I=X_1,\dots, X_n$, described by a diagram $(\Sigma; C_1,\dots, C_n)$. Let $X_{i,j}$ be a subsection and $\phi: \boundary X_{i,j} \to \boundary X_{i,j}$ a homeomorphism respecting the Heegaard splitting of $\partial X_{i,j}$. Then, $(\Sigma; C_1,\dots,C_i,\allowbreak \phi(C_{i+1}),\allowbreak \phi(C_{i+2}),\dots, \phi(C_{j-1}),\allowbreak\phi(C_j), C_{j+1}, C_{j+2},\allowbreak\dots, C_n)$ is a multisection diagram for $(X- X_{i,j})\cup_{\phi} X_{i,j}$.
\end{proposition}

Next, we introduce operations which allow us to remove and introduce additional sectors. We begin with the process of removing a sector.

\begin{definition}
Let $X$ be a multisected 4-manifold with sectors $X_I=X_1,\dots, X_n$. Suppose that the sectors $X_i$ and $X_{i+1}$ have the property that $X_i \cup_{H_{i,i+1}}X_{i+1} \cong \natural^k S^1 \times B^3$ for some $k$. Then the \emph{contraction} of $X_I$ along $H_{i,i+1}$ is the $n-1$-section with sectors $X_I'$, where for $0 < k < i$, $X_k' = X_k$, for $k=i,$ $X_i' = X_i \cup X_{i+1}$, and for $i< k \leq n-1$, $X_i' = X_{i+1}$. 
\end{definition}

The inverse procedure of contraction introduces additional sectors to a multisection. While this may seem undesirable at first, this operation reduces the complexity of each individual sector, which has certain advantages. In fact, by iterating this procedure, any multisection can be turned into a thin multisection. The proof of this fact is delayed until Proposition \ref{prop:thinMultisections}.

\begin{definition}\label{def:expansion}
Let $X$ be a multisected 4-manifold with sectors $X_I=X_1,\dots, X_n$. Suppose that $X_i$ admits a bisection $X_i = A \cup_{H'} B$ inducing the original Heegaard splitting $\partial X_i=H_{i-1,i} \cup_\Sigma H_{i,i+1}$. Then the \emph{expansion} of $X_I$ along $H'$ is the $(n+1)$- section with sectors $X_I'$, where for $0 < k < i$, $X_k' = X_k$, $X_i' = A$, $X_{i+1}' = B$, and for $i+1<k \leq n+1$, $X_k' = X_{i-1}$.
\end{definition}

We remark that one can always trivially expand a multisection by adding an additional sector which is just the product neighbourhood of some $H_{i,i+1}$, though this will not reduce the complexity of any sector. Sectors created in this fashion are superfluous, and the following lemma allows us to remove them.

\begin{lemma}
\label{lem:removeRedundantSector}
Let $n\geq 3$. Suppose $X$ admits an $n$-section of genus $g$ and that for some $i$, $k_i = g$. Then, $X$ admits an $(n-1)$-section of genus $g$.
\end{lemma}

\begin{proof}
By definition, $k_i = g$ implies that $X_i \cong \natural^{g} S^1 \times B^3$. Moreover, $\partial X_i = H_{(i-1),i} \cup H_{i,(i+1)}$ where both $H_i$ and $H_{i+1}$ are handlebodies of genus $g$, and so $X_i$ is diffeomorphic to $H_i \times I$. Note that $\partial X_{i-1} \cap X_i = H_{(i-1),i}$, and so $X_{i-1}' = X_{i-1} \cup_{H_{(i-1),i}} X_{i}$ is topologically just a collar on $X_{i-1}$, i.e., still a 4-dimensional 1-handlebody. The boundary of $X_i'$ has the requisite Heegaard splitting $\partial X_i'=H_{(i-1),(i-2)} \cup H_{i,(i+1)}$, where the assumption that $n \geq 3$ ensures that these are not the same handlebody. Therefore,  the decomposition $X = X_1 \cup X_2 \cup \dots\cup X_{i-1}' \cup X_{i+1} \cup \dots \cup X_{n}$ is a genus $g$ multisection of $X$.
\end{proof}

\section{Multisection genus}
 
In this section, we study the genus of a multisection, which is a natural measure of complexity. First, we classify all 4-section diagrams on the torus, along with their corresponding cross-sections. If a 4-manifold, $X$, admits a genus one 4-section in which some $k_i=1$, then by Lemma \ref{lem:removeRedundantSector} $X$ admits a genus one trisection. Genus one trisections are easy to classify \cite{GayKir16}: these have diagrams described in Figure \ref{fig:TorusTrisections}. We therefore focus on manifolds admitting $(1;0)$- 4-sections; a standard diagram for each such manifold is illustrated in Figure \ref{fig:4sectionstorus}.

\begin{figure}[ht]
    \centering
    \includegraphics[width=0.7\textwidth]{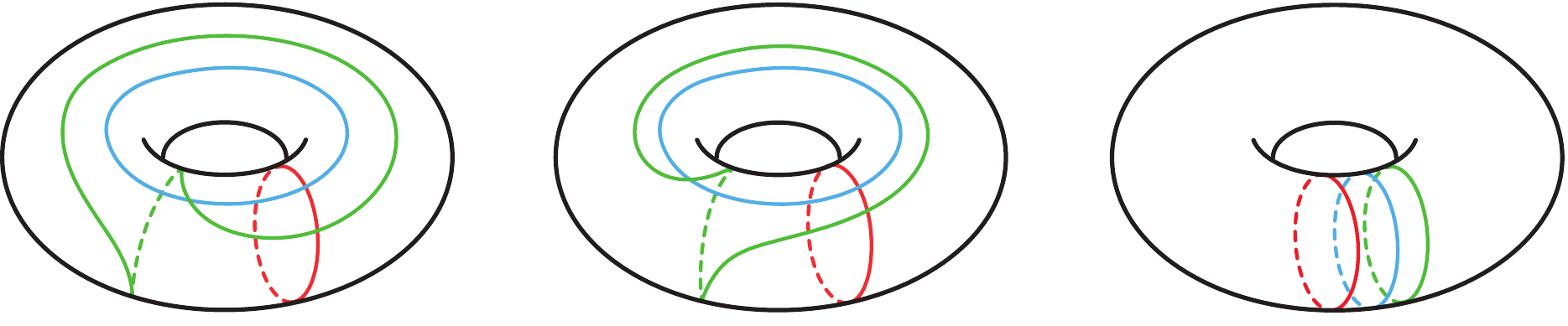}
    \put(-285,-20){(a)}
    \put(-170,-20){(b)}
    \put(-55,-20){(c)}
    \caption{The possible  $(1;0)$- trisection diagrams. These are: (a) $\cp$, (b) $\cpbar$, and (c) $S^1\times S^3$. All cross-sections are either $S^3$ or $S^1\times S^2$. }%
    \label{fig:TorusTrisections}
\end{figure}

\begin{proposition}\label{prop:genusoneclassification}
Suppose that a 4-manifold $X$ admits a $(1;0)$- $4$-section. Then this 4-section may be described by one of the diagrams in Table \ref{tab:4sections}. In particular, $X$ is diffeomorphic to $S^2\times S^2$, $S^2\simtimes S^2$, $\cp\#\cp$, or $\cpbar\#\cpbar$. 
\end{proposition}

\begin{figure}[ht]
    \centering
    \includegraphics[width=\textwidth]{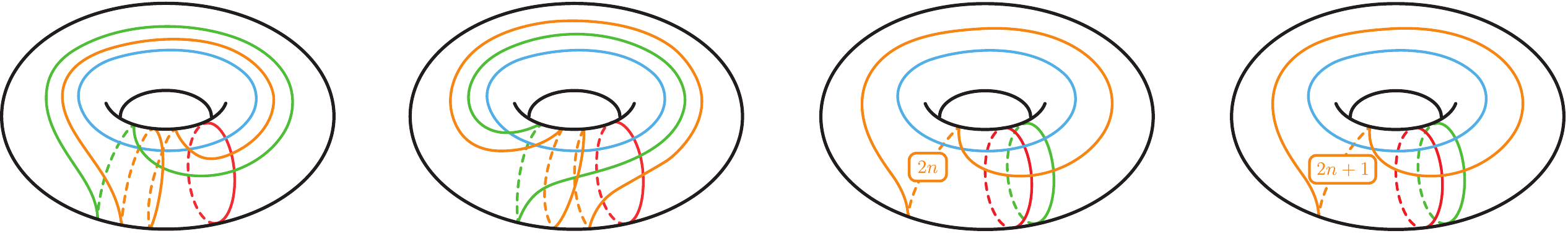}
    \put(-424,-20){(a)}
    \put(-301,-20){(b)}
    \put(-178,-20){(c)}
    \put(-55,-20){(d)}
    \caption{One possible $(1;0)$- 4-section diagram for each of the standard manifolds ($\alpha$, $\beta$, $\gamma$, and $\delta$ are red, blue, green, and orange respectively). These are: (a) $\cp\#\cp$, (b) $\cpbar\#\cpbar$, (c) $S^2\times S^2$, and (d) $S^2\simtimes S^2$. }%
    \label{fig:4sectionstorus}
\end{figure}

\proof
Suppose that $X$ admits a $(1;0)$- 4-section, with diagram \allowbreak $(\Sigma_1;\alpha,\beta,\allowbreak\gamma,\delta)$. Since $k_i=0$ for all $i$, each pair of adjacent curves are geometrically dual. Fix a basis of $H_1(\Sigma_1;\mathbb{Z})$. Then without loss of generality, we may assume that $\alpha$ and $\beta$ are the $(1,0)$ and $(0,1)$ curves on $\Sigma_1$. Assume that $\gamma=(\gamma_1,\gamma_2)$ and $\delta=(\delta_1,\delta_2)$. Since the $(p,q)-$ and $(r,s)-$ curves intersect $|ps-rq|$ times, we must have $|\gamma_1|=|\delta_2|=1$. Without loss of generality, we will take $\gamma_1=\delta_2=1$. Lastly, since $\gamma$ and $\delta$ are dual, we find that: 
\[|\gamma_1\delta_2-\gamma_2\delta_1|=|1-\gamma_2\delta_1|=1\]

The possibilities for the curves are summarized in the table below, and we can use the algorithm in Proposition \ref{prop:handledecomposition} to identify the corresponding 4-manifold. Note that the cross-section manifolds admit genus one Heegaard splittings, so are necessarily lens spaces. This completes the classification of genus one 4-sections. 

\begin{center}
\begin{table}[ht]\label{tab:4sections}
\renewcommand{\arraystretch}{1.5}%
\begin{tabular}{|c|c|c|c||c|c||c|}
\hline
$\alpha$ & $\beta$ & $\gamma$ & $\delta$ & $(\Sigma_1;\alpha,\gamma)$ & $(\Sigma_1;\beta,\delta)$ & $(\Sigma_1;\alpha,\beta,\gamma,\delta)$ \\
\hline
$(1,0)$ & $(0,1)$ & $(1,2n)$ &  $(0,1)$ & $L(2n,1)$ & $S^1\times S^2$ & $S^2\times S^2$ \\
\hline
$(1,0)$ & $(0,1)$ & $(1,0)$ &  $(2n,1)$ & $S^1\times S^2$ & $L(2n,1)$ & $S^2\times S^2$ \\
\hline
$(1,0)$ & $(0,1)$ & $(1,2n+1)$ &  $(0,1)$ & $L(2n+1,1)$ & $S^1\times S^2$ & $S^2\simtimes S^2$ \\
\hline
$(1,0)$ & $(0,1)$ & $(1,0)$ &  $(2n+1,1)$ & $S^1\times S^2$ & $L(2n+1,1)$ & $S^2\simtimes S^2$ \\
\hline
$(1,0)$ & $(0,1)$ & $(1,2)$ &  $(1,1)$ & $\RP^3$ & $S^3$ & $\CP^2\#\CP^2$\\
\hline
$(1,0)$ & $(0,1)$ & $(1,1)$ &  $(2,1)$ & $S^3$ & $\RP^3$ & $\CP^2\#\CP^2$ \\
\hline
$(1,0)$ & $(0,1)$ & $(1,-2)$ &  $(-1,1)$ & $\RP^3$ & $S^3$ & $\overline{\CP}^2\#\overline{\CP}^2$ \\
\hline
$(1,0)$ & $(0,1)$ & $(1,-1)$ &  $(-2,1)$ & $S^3$ & $\RP^3$ & $\overline{\CP}^2\#\overline{\CP}^2$ \\
\hline
\end{tabular}
\vspace{5mm}
\caption{The complete list of $(1;0)$ 4-sections.}
\end{table}
\end{center}
\vspace{-10mm}

\qed

\begin{remark}
A reader familiar with trisections will note that the manifolds described by $(1;0)-$ 4-sections are exactly the simply connected manifolds admitting genus two trisections. Indeed, if a 4-manifold $X$ admits a $(1;0)$- 4-section diagram, then by Proposition \ref{prop:turnIntoTri} we may convert this to a genus two trisection diagram. By Meier and Zupan's \cite{MeiZup17} classification of genus two trisections, $X$ must be one of the standard manifolds in Proposition \ref{prop:genusoneclassification} (there can be no $S^1\times S^3$ summands if $X$ is simply connected). 
\end{remark}

\subsection{Multisection genus and connected sum}

The question of whether trisection genus is additive under connected sum is a difficult question and, in particular, implies the Poincar\'e conjecture. In this section, we will show that multisection genus is far from additive. One reason for this is that multisections can be glued together along sectors diffeomorphic to 4-balls in order to produce multisections of a connected sum, as in Figure \ref{fig:connectedSum}. 

\begin{lemma}
Let $X$ and $X'$ be multisected 4-manifolds with sectors $X_I=X_1,\dots,X_n$ and $X_I'=X_1',\dots, X_{n'}'$, respectively, and suppose that each multisection is of genus $g$. Further, suppose that each multisection admits a sector diffeomorphic to a 4-ball. Then, $X \# X'$ admits an $(n+n'-2)$- section of genus $g$. 
\end{lemma}

\begin{proof}
Suppose $X_i \subset X_I$ and $X_j' \subset X_I'$ are the sectors diffeomorphic to 4-balls. We may choose these particular balls to perform the connected sum operation. In order for the resulting manifold to admit a multisection, the connected sum map must send $H_{i-1, i}$ to $H_{j-1,j}'$ and $H_{i,i+1}$ to $H_{j,j+1}$. Since these pairs of handlebodies form genus $g$ Heegaard splittings of $S^3$, which are unique up to isotopy \cite{Wal68}, the map may be arranged to satisfy this condition after an isotopy.
\end{proof}

We can also iterate this process. If a manifold admits an $n$-section of genus $g$ with at least two sectors diffeomorphic to 4-balls, then performing this connected sum operation leaves 4-ball sectors left over. We omit the proof of the following lemma, which only differs from the previous one by this observation.

\begin{lemma}\label{lem:arbitraryConnectedSum}
Let $X_1,X_2,\dots,X_n$ be multisected 4-manifolds, with sectors $X_{I_1}\allowbreak,X_{I_2},\dots,X_{I_n}$ respectively, and suppose that each multisection is of genus $g$. Further, suppose that each multisection contains at least two sectors diffeomorphic to $B^4$. Then, the manifold $\#^{a_1}X_1 \#^{a_2}X_2 \dots \#^{a_n}X_n$ also admits a genus $g$ multisection.
\end{lemma}

\begin{figure}
    \centering
    \includegraphics[scale=.15]{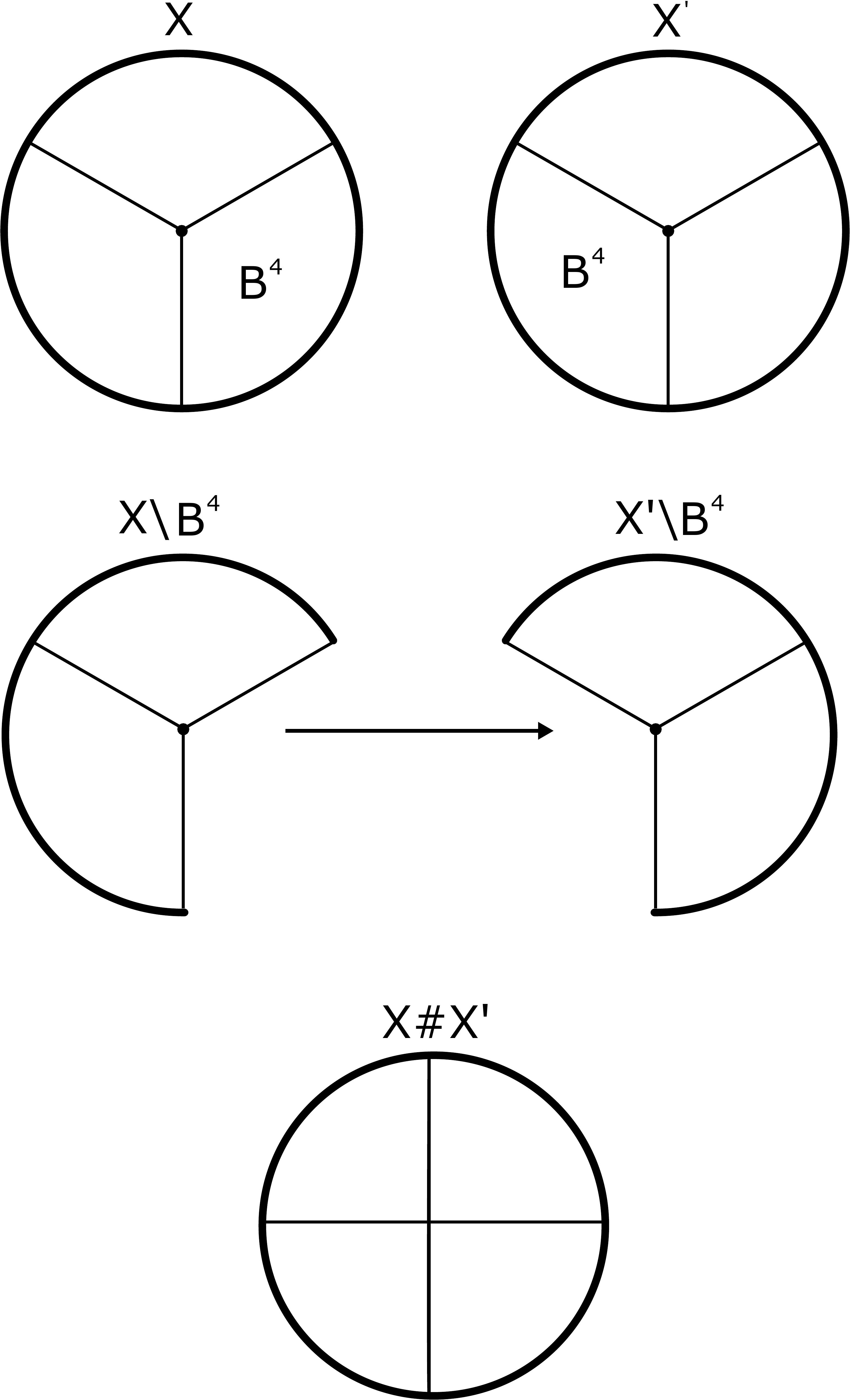}
    \caption{If $X$ and $X'$ admit multisections where each sector is a 4-ball, then $X \# X'$ admits a multisection obtained by cutting out each of these 4-balls and gluing the remaining pieces together.}
    \label{fig:connectedSum}
\end{figure}

Since $\cp$ and $\cpbar$ admit $(1;0)$- trisections, and by Proposition \ref{prop:genusoneclassification}, $S^2\times S^2$ admits a $(1;0)$- 4-section, the following proposition is immediate.

\standardManifoldsAreGenusOne

On the other hand, the trisection genus of these manifolds is unbounded. In fact, these examples fit into a more general framework: if a smooth 4-manifold admits a decomposition without 1- or 3-handles, then it is known to admit a trisection with two sectors diffeomorphic to 4-balls \cite{MeiSchZup16}. While this may seem like a restrictive class of manifold, it is conjectured that all simply connected 4-manifolds admit such a decomposition. The following question addresses this conjecture.

\begin{question}
Does there exist a simply connected 4-manifold, $X$, such that for $n\in \mathbb{N}$ the multisection genus of $\#^n X$ is unbounded?
\end{question}

By Lemma \ref{lem:arbitraryConnectedSum}, an affirmative answer would produce a simply connected 4-manifold which requires either 1- or 3-handles. In light of \cite{BCKM19}, such an example would also produce a simply connected manifold which is not an irregular 3-fold branched cover over $S^4$.

\section{Corks and 4-sections}\label{sec:corkTwists}

\subsection{Corks twists as subsection operations}

For our purposes, a \emph{cork} is a pair $(A, \tau)$, where $A$ is a contractible compact 4-manifold and $\tau$ is an involution on its boundary, usually called a \emph{cork twist}. We will often suppress the involution from this notation. If $A \subset X^4$ then we will call the manifold $X_A = (X - A) \cup_\tau A$ a cork twist of $X$ about $A$. Corks are at the center of the failure of the $h$-cobordism theorem in dimension 4, as well as exotic behavior of 4-manifolds, and this is illustrated by the following theorem.

\begin{theorem}[\cite{CurFreHsiSto96}, \cite{Mat96}]\label{thm:pairsAreCorkTwists}
Every pair of simply connected, orientable, closed 4-manifolds which are homeomorphic but not diffeomorphic are related by a cork twist, i.e., are of the form $X$ and $X_A$ for some cork $A \subset X$. Moreover, the cork $A$ may be chosen so that $X$ and $X -A$ are 2-handlebodies.
\end{theorem}

We now investigate the effect of a cork twist on a 4-section diagram. 

\corkCurves

\begin{proof}
Suppose that $X$ and $X'$ are homeomorphic but not diffeomorphic. By Theorem \ref{thm:pairsAreCorkTwists}, there exists a cork, $A$, such that $A$ and $X \backslash A$ are both 2-handlebodies, and $X_A$ is diffeomorphic to $X'$. By Corollary \ref{thm:2HandlebodyHasBisection}, both $X$ and $X - A$ admit bisections. After stabilization, these bisections can be glued to form a 4-section of $X$. Let $X=X_1 \cup X_2 \cup X_3 \cup X_4$ be such a quadrisection with $X_{1,2} = A$ and $X_{3,4} = X -A$.

By Lemma \ref{prop:reglueAfterStab}, after sufficiently many stabilizations, the cork twist $\tau$ can made compatible with the 4-section, in sense that $\tau$ fixes $\Sigma$ and extends across $H_{1,2}$ and $H_{2,3}$. By the equivariant disk theorem \cite{MeeYau81}, there exist cut systems, $C_1$ and $C_3$, for $H_{4,1}$, and $H_{2,3}$, respectively so that $\tau(C_1) = C_1$ and  $\tau(C_3) = C_3$. The effect of $\tau$ on $C_2$ follows from proposition \ref{prop:CurvesAfterGluing}.
\end{proof}

We remark that in \cite{AkbMat98}, the authors show that the cork and its complement in Theorem \ref{thm:pairsAreCorkTwists} can be chosen to be compact Stein manifolds. Using Corollary \ref{cor:steinMfldsHaveBisections}, the previous proof admits a straightforward modification where the bisections involved are Stein manifolds. 

\subsection{The cork twist on a Mazur manifold}\label{sec:mazurex}

In this section, we produce an explicit example of a bisection of a cork, together with the effect of the cork twist. Consider the 4-manifold, $W$, given by the Kirby diagram in Figure \ref{fig:mazManWithHeegSplit}. This manifold was first described by Mazur in \cite{Maz61}, who showed that $W$ is contractible. A rotation about the axis of symmetry of the link in the Kirby diagram induces an involution, $\tau$, on the boundary of this manifold. Akbulut showed in \cite{Akb91} that this involution does not extend to the whole 4-manifold, and so $(W, \tau)$ is indeed a cork.

The manifold $W$ is obtained by attaching a $2$-handle to $\partial(S^1 \times B^3) = S^1 \times S^2$. In addition to the handle decomposition, in Figure $\ref{fig:mazManWithHeegSplit}$ we also include a Heegaard surface, $\Sigma$, splitting this copy of $S^1\times S^2$ into two handlebodies. We will label the ``interior'' handlebody $H_\alpha$, and the ``exterior'' handlebody $H_\beta$, so that $S^1\times S^2=H_\alpha\cup_\Sigma H_\beta$. On $\Sigma$ we also see the curve which bounds a disk in both $H_\alpha$ and $H_\beta$ (in purple). The 2-handle projects onto $\Sigma$ as $\gamma$ (in green), and this projection has the additional property that its surface framing is equal to $0$. Since $\gamma$ is dual to $H_\beta$, pushing $\gamma$ into $H_\beta$ and doing surface-framed surgery is a handlebody, $H_\gamma$. Sliding the curves in a cut system for $H_\beta$ off of $\gamma$ via this dual curve produces the cut system $C_\gamma$ in Figure \ref{fig:mazManCorkTwist}. 

The duality condition between the $H_\beta$ and $\gamma$ also implies that $\Sigma$ is a Heegaard surface for $\partial W$. This surface is fixed setwise with respect to $\tau$ and, in particular, $\tau$ induces an involution on $\Sigma$. By Proposition \ref{prop:CurvesAfterGluing}, the effect of doing a cork twist on this bisection is replacing $C_\beta$ with $C_\beta'$, which is also illustrated in Figure \ref{fig:mazManCorkTwist}.

\begin{figure}
    \centering
    \includegraphics[width=0.9\textwidth]{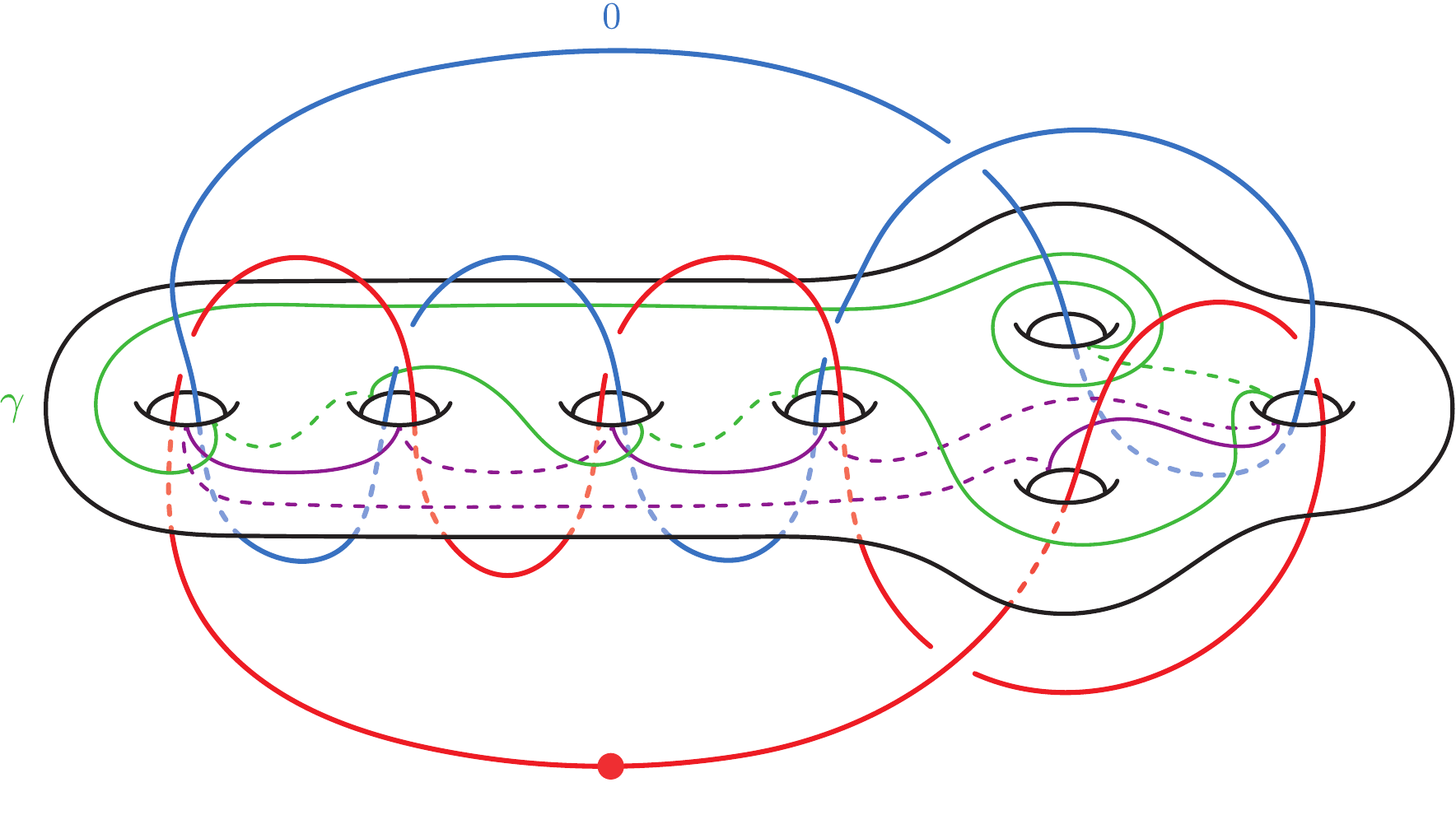}
    \caption{An equivariant Heegaard splitting for the boundary of the Mazur manifold.}
    \label{fig:mazManWithHeegSplit}
\end{figure}
 
In fact, Akbulut showed that $W$ may be embedded in $E(2) \# \cpbar$, and that cutting out and regluing $W$ by $\tau$ changes the smooth structure on $E(2) \# \cpbar$. This was later generalized by Bi\v{z}aca and Gompf \cite{BizGom96}, who demonstrated an embedding of $W$ in $E(n) \#^2 \cpbar$, so that cutting and regluing $W$ by $\tau$ also changes the smooth structure. In their decomposition of $E(n) \#^2 \cpbar$, \cite[page 477]{BizGom96} the complement of $W$ is built without handles of index $3$ or $4$. Therefore, by Theorem \ref{thm:2HandlebodyHasBisection} both $W$ and $(E(n)\#^2\cpbar)-W$ admit bisections which can (after some stabilizations) be glued together. The involution in Figure \ref{fig:mazManCorkTwist} naturally extends to the stabilized bisections by performing the stabilization in an equivariant way. Thus, we obtain the following proposition.

\begin{proposition}
There are infinitely many exotic pairs of 4-manifolds, $X$ and $X'$, satisfying the following properties:
\begin{enumerate}
    \item $X$ and $X'$ have 4-section diagrams $(\Sigma;D_1, D_2, D_3, D_4)$ and $(\Sigma;D_1, D_2, D_3, D_4')$, respectively;
    \item $D_2$, $D_3$, and $D_4$ are some stabilization of the cut systems $C_2$, $C_3$, and $C_4$, in Figure \ref{fig:mazManCorkTwist}, respectively;
    \item $D_4'$ is some stabilization of the cut system $C_4'$ in Figure \ref{fig:mazManCorkTwist}.
\end{enumerate}
\end{proposition}

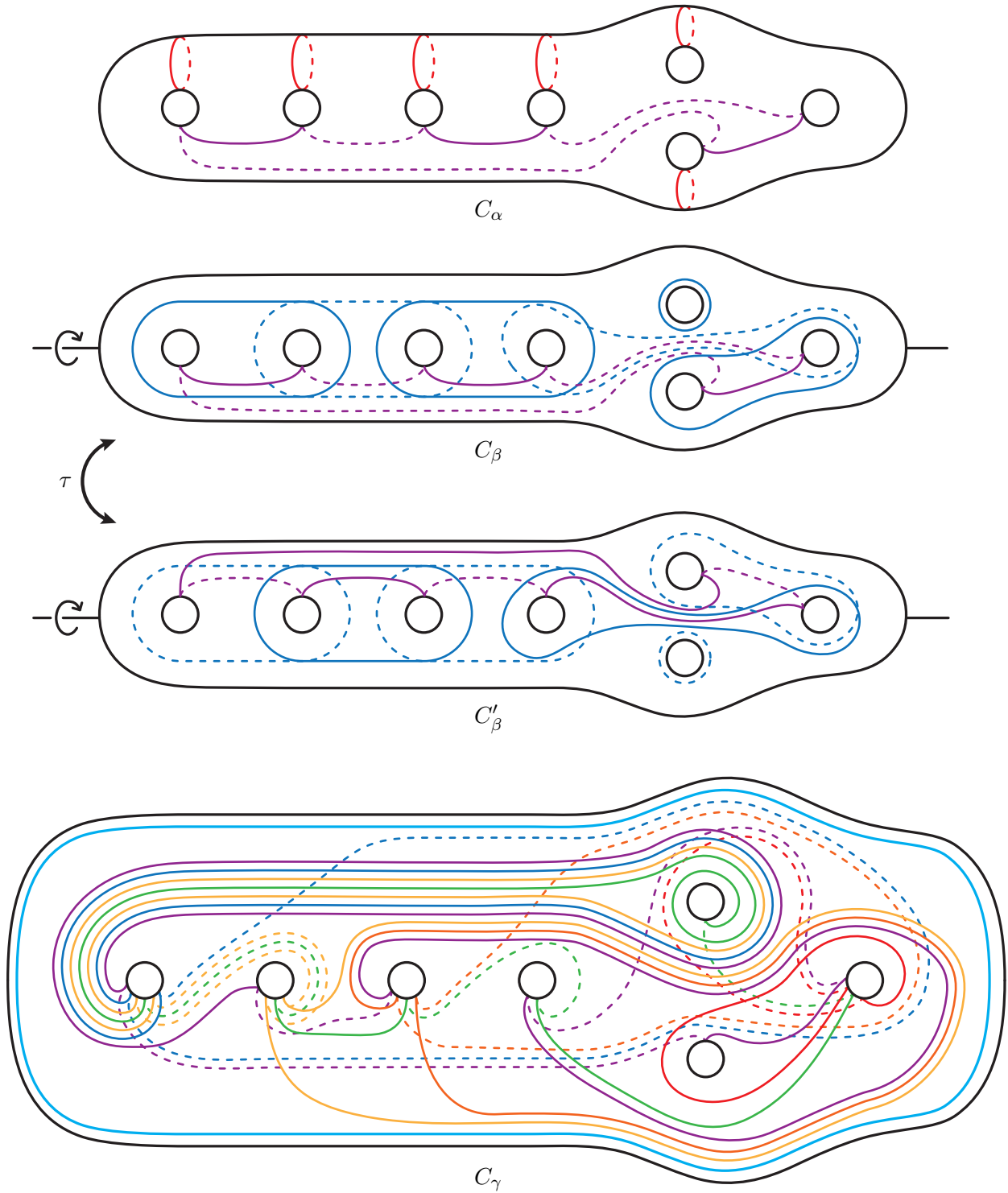
\begin{figure}
    \centering
    \includegraphics[width=\textwidth]{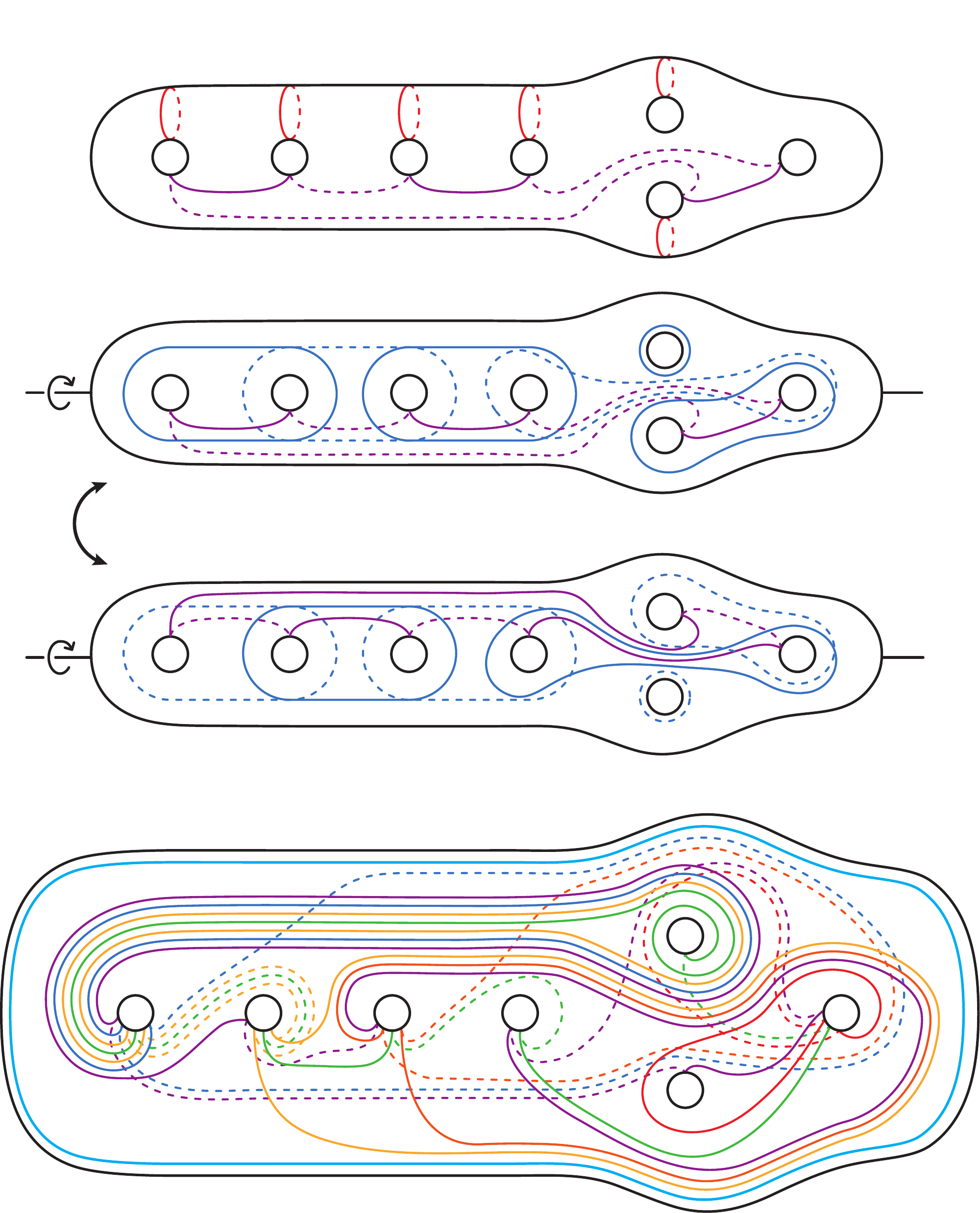}
    \put(-250,455){$C_\alpha$}
    \put(-250,342){$C_\beta$}
    \put(-250,217){$C_\beta'$}
    \put(-250,0){$C_\gamma$}
    \put(-445,328){$\tau$}
    \caption{The three cut systems for a bisection diagram of the Mazur manifold: $C_\alpha$, $C_\beta$, and $C_\gamma$. The curves in $C_\gamma$ are colored differently only for visual clarity.}
    \label{fig:mazManCorkTwist}
\end{figure}

\subsection{The failure of Waldhausen's theorem for 4-sections}\label{sec:waldhausen}

Another interesting question in trisection theory is whether the analogue of Waldhausen's theorem \cite{Wal68} holds for trisections of $S^4$, i.e., whether every trisection of $S^4$ is a stabilization of the $(0;0)$- trisection. Here, we answer the corresponding question for multisections in the negative. This question is only interesting in the case that the cross-section manifolds are identical; otherwise the multisections could not possibly be diffeomorphic. Our main tool will be the (infinitely many) exotic Mazur manifolds constructed in \cite{HayMarPic19}, i.e., pairs of compact contractible 4-manifolds built with a single 1- and 2-handle, that are homeomorphic but not diffeomorphic.

\begin{proposition}
There are infinitely many pairs of non-diffeomorphic 4-sections of $S^4$ with the same 3-manifold cross-sections.
\end{proposition}

\begin{proof}
Let $Z$ and $Z'$ be a pair of non-diffeomorphic Mazur manifolds with $\partial Z=\partial Z'=Y$. Since $Z$ and $Z'$ are 2-handlebodies, they admit bisections, and up to stabilization, the doubles $D(Z)=Z\cup Z$ and $D(Z')=Z'\cup Z'$ admit 4-sections, say of genus $g$. The cross-section manifolds are $Y$ and $\#^gS^1\times S^2$ (the double of the middle handlebody). Since both $Z$ and $Z'$ are of Mazur type, their doubles are $S^4$. On the other hand, these two 4-sections of $S^4$ cannot be diffeomorphic, since this would restrict to a diffeomorphism between $Z$ and $Z'$. Applying this construction to the pairs of exotic Mazur manifolds from \cite{HayMarPic19} completes the proof.
\end{proof}

\section{Multisections of \texorpdfstring{$E(n)_{p,q}$}{E(n)p,q}}\label{sec:EllipicFibrations}

\subsection{Elliptic fibrations and \texorpdfstring{$\mathbf{E(1)}$}{E1}} 

Recall that a complex surface $S$ is a holomorphic elliptic fibration if there is a holomorphic map $p:S \to C$ to a complex curve, $C$, such that for all $t \in C$, $p^{-1}(t)$ is an elliptic curve, i.e. topologically a torus. For a smooth, closed, oriented 4-manifold $X$, we say that $p:X \to C$ is a smooth elliptic fibration if for all $t \in C$, the fiber $p^{-1}(t)$ has a neighborhood modeled on a holomorphic elliptic fibration. More precisely, $p^{-1}(t)$ has a neighborhood, $U$, and an orientation preserving diffeomorphism, $\phi: U \to S$, to an elliptic fibration, $S$, such that $\phi$ commutes with the fibration maps. 

In this section, we will be focused on smooth elliptic fibrations whose base surface, $C$, is topologically a sphere. Of particular importance in this class is the elliptic fibration $p:\cp \#^9 \cpbar \to S^2$. We call this manifold, together with its fibration structure, $E(1)$.  Let $F,F' \subset E(1)$ be regular fibers of this fibration, and let $\nu(F)$ and $\nu(F')$ be their regular neighbourhoods. By choosing these neighbourhoods sufficiently small, we can arrange so that they contain no critical fibers. We therefore obtain a decomposition of $E(1)$ as $\nu(F) \cup_{T^3} A \cup_{T^3} \nu(F')$, where $A$ is the region of $E(1)$ projecting to the annulus $S^2 - (p(\nu(F)) \cup p(\nu(F')))$.

We parameterize $\partial \nu(F) \cong T^3$ as $S^1 \times S^1 \times S^1$, where the first two coordinates span a regular fiber, and the third coordinate corresponds to the circle which bounds a disk in $\nu(F)$. Using $F$ as a reference fiber, $E(1)$ has twelve critical fibers whose vanishing cycles, in order, alternate between the loops $S^1 \times {\text{\{pt\}}} \times {\text{\{pt\}}}$ and $ {\text{\{pt\}}} \times S^1 \times {\text{\{pt\}}}$. The effect of expanding a subset of the fibration across a critical fiber is a 2-handle addition along the vanishing cycle with framing one less than the fiber-surface framing. We can therefore describe $A$ by the self-cobordism of $T^3$ consisting of twelve 2-handle additions, pictured in Figure \ref{fig:E1HandleDecomp}.

\begin{figure}
    \centering
    \includegraphics[scale=.3]{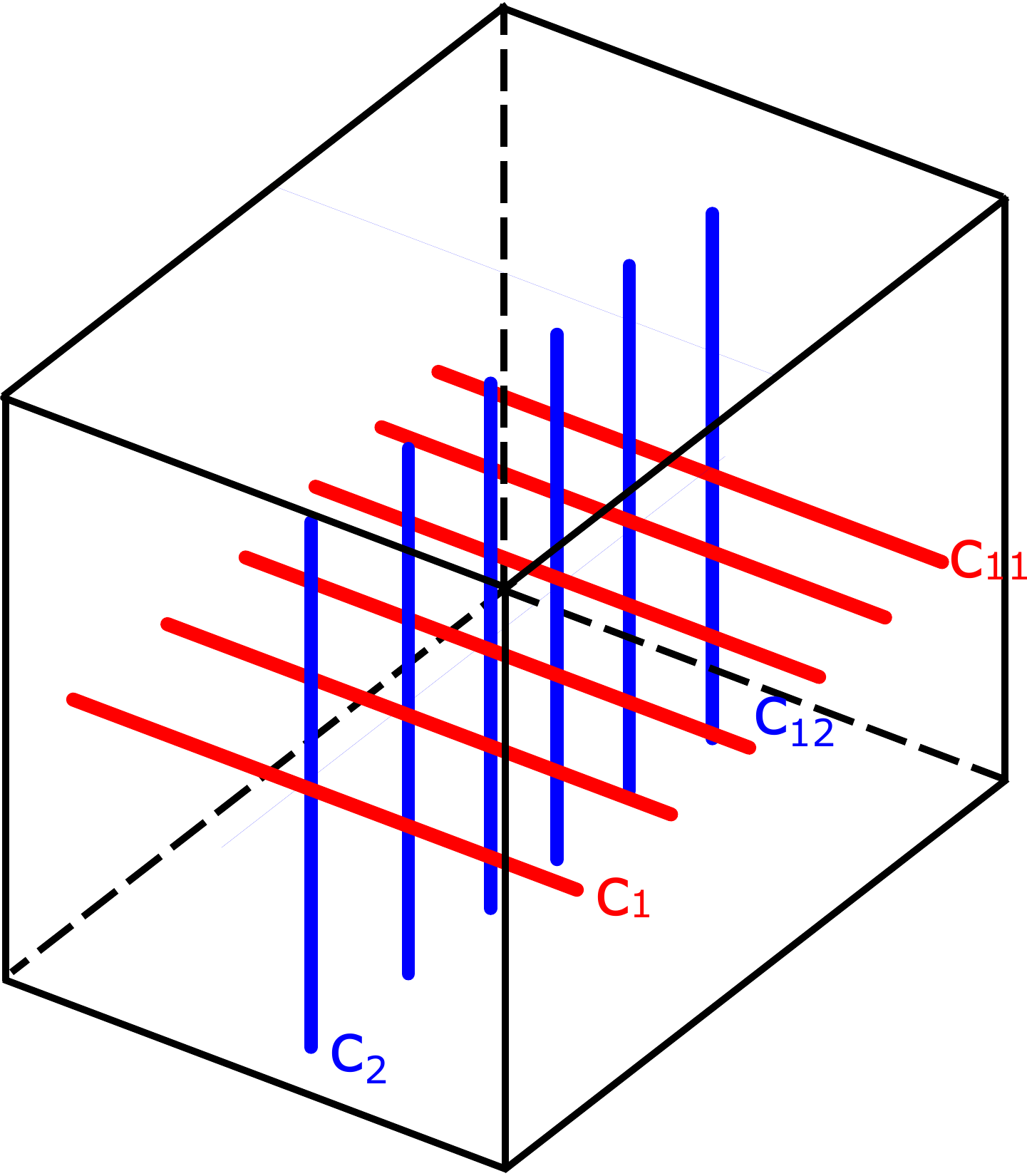}
    \caption{The self cobordism of $T^3$ corresponding to the complement of two regular fibers of $E(1)$ consists of twelve layered 2-handles. All of these 2-handles are $-1$- framed with respect to the coordinate $T^2$ spanned by a red and a blue curve.}
    \label{fig:E1HandleDecomp}
\end{figure}
 
We now construct a multisection of $E(1)$. We begin with $\nu(F)$, which is diffeomorphic to $T^2 \times D^2$ and can be obtained by attaching a $0$-framed 2-handle to $T^3 \times I$ along the loop $ {\text{\{pt\}}} \times {\text{\{pt\}}} \times S^1$. This 2-handle attachment is compatible with the Heegaard splitting of $T^3$, since the projection of this curve to the Heegaard surface, $\Sigma$, is dual to the cut systems $D_1$ and $D_3$ in Figure \ref{fig:T2xD2}. The reader may verify this fact by noting that the projection of the 2-handle attaching circle is the red curve in $D_2$. In particular, Figure \ref{fig:T2xD2} is a bisection diagram for $T^2\times D^2$; where $D_1$ and $D_3$ is the usual Heegaard diagram of $T^3$.

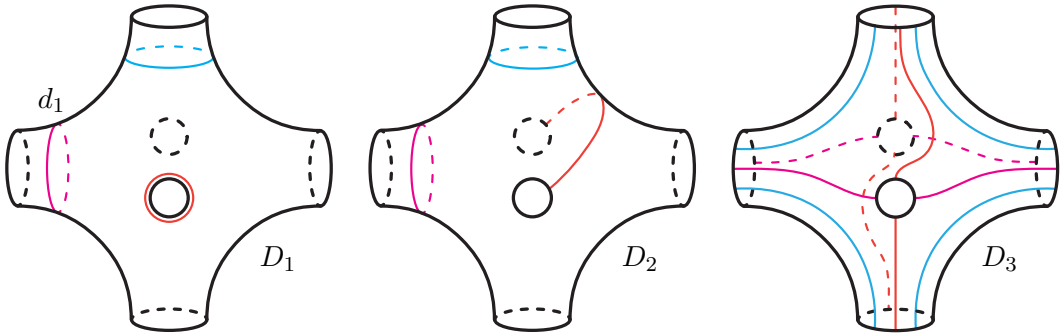
\begin{figure}
    \centering
    \includegraphics[width=0.85\textwidth]{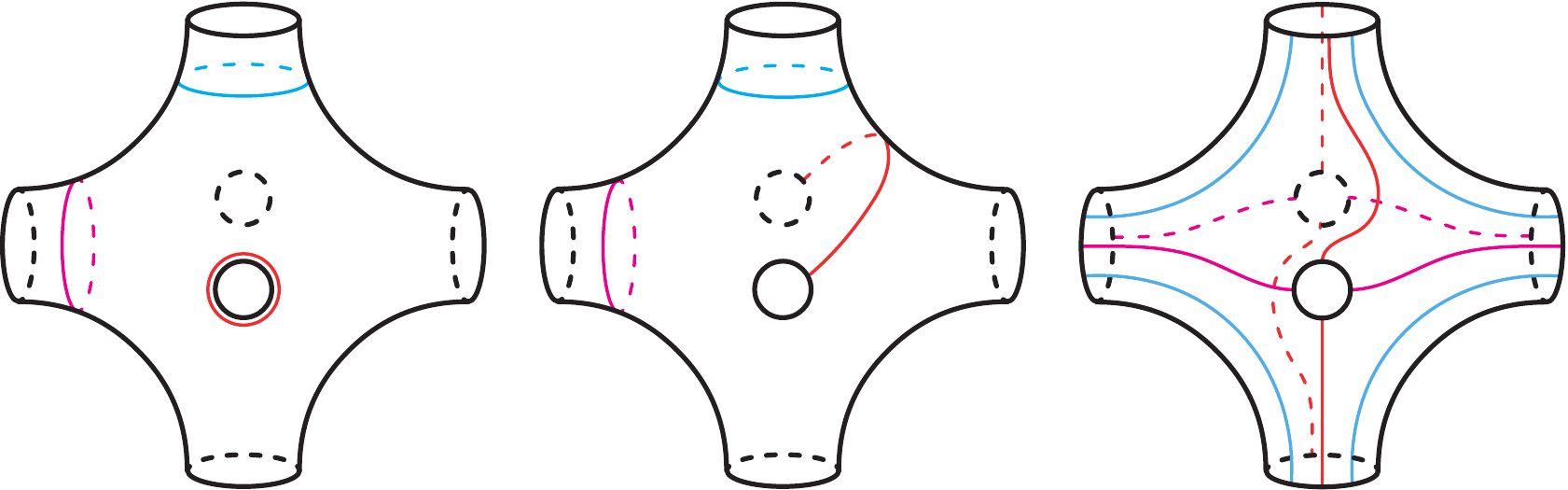}
    \put(-30,25){$D_{3}$}
    \put(-166,25){$D_{2}$}
    \put(-303,25){$D_{1}$}
    \put(-387,85){$d_1$}
    \caption{A bisection diagram of $T^2 \times D^2$. The different colors in each cut system are only for visual clarity.}
    \label{fig:T2xD2}
\end{figure}

Next, we incorporate the cobordism, $A$; to do so, we will attach each 2-handle $c_1,\dots,c_{12}$ in Figure \ref{fig:E1HandleDecomp} in order of increasing index. The projection of $c_1$ onto $\Sigma$ has an obvious dual curve $d_1$ (labelled in $D_1$). Performing a left-handed Dehn twist of $c_1$ about $d_1$ produces a curve $c_1^p \subset \Sigma$ isotopic to $c_1$ in $T^3$, but whose framing induced by $\Sigma$ is one less than the framing induced by the fiber surface $S^1 \times S^1 \times {\text{\{pt\}}}$. Therefore, replacing the curve $d_1$ in the Heegaard diagram for $T^3$ by $c_1^p$ corresponds to adding a $-1$- framed 2-handle along $c_1$, as desired. The resulting cut system is pictured in the top left of Figure \ref{fig:CobordismMultisection}, and labelled $C_1$.

We next proceed to $c_2$; we would like to push it in front of the 2-handle we have already attached. Sliding $c_2$ over the framed curve $c_1$ corresponds to a right-handed Dehn twist of $c_2$ about $c_1$ in the fiber $S^1\times S^1\times\{\text{pt}\}$. Note that such a slide preserves the surface framing. The projection of this twisted curve to the Heegaard surface is dual to the curve $d_2\subset \Sigma$ (labelled in $C_1$), and so, correcting the framing as before, we obtain a curve $c_2^p$. Replacing $d_2$ with $c_2^p$ produces the cut system $C_2$ in Figure \ref{fig:CobordismMultisection}.

We proceed similarly with the remaining 2-handles. Sliding $c_i$ over the curves $c_1,\dots,c_{i-1}$ corresponds to performing a right handed Dehn twist of $c_i$ about $c_{i-1},c_{i-2},\dots,c_2$, and $c_1$, in this order. After these slides, $c_i$ lies in front of the other curves and so can be projected to the front of $\Sigma$. With coordinates as before, the result of sliding $c_i$ in front of the previous curves produces the $(0,1,0)$- curve if $i \equiv 0\! \mod 3$, the $(1,0,0)$- curve if $i \equiv 1\! \mod 3$, and the $(1,1,0)$- curve if $i\! \equiv 2 \mod 3$. 

For $i \geq 3$, the result of projecting $c_i$ to the Heegaard surface is dual to the curve $c_{i-2}^p$. We can frame the curve appropriately by performing a left handed Dehn twist of $c_i$ about $c_{i-2}^p$, producing a curve $c_i^p$. Replacing $c_{i-2}^p$ by  $c_i^p$ gives the next handlebody in the sequence. The remainder of Figure \ref{fig:CobordismMultisection} is the result of repeatedly applying this procedure. By construction, the handlebodies determined by the cut systems $D_3$ and $C_i$ define a Heegaard splitting for the torus bundle over $S^1$ with monodromy obtained by taking the first $i$ terms of the expression $(\tau_{\alpha}\tau_{\beta})^6$. In particular, the well known relation $(\tau_{\alpha}\tau_{\beta})^6=1$ in the mapping class group of the torus implies that the final cut system, $C_{12}$, together with $D_3$, is a Heegaard diagram for $T^3$.

Combining Figures \ref{fig:T2xD2} and \ref{fig:CobordismMultisection} we now have a subsection of $E(1)-\nu(F')$ in a forthcoming multisection for $E(1)$. To extend this multisection to $\nu(F')$, we fill in the remaining $T^3$ boundary with the bisection of Figure \ref{fig:T2xD2}, by adding the cut system $D_2$. More explicitly, a complete multisection diagram of $E(1)$ is given by the ordered cut systems $(D_3,D_2, D_1, C_1, C_2, \dots C_{11}, C_{12}=D_1, D_2)$. Note that this is a thin $(3;2)$ 16-section, and so $\chi(E(1))=12$, as expected. We remark that the cut systems $C_i$ and $C_{i+1}$ can be contracted, so $E(1)$ admits a 10-section of genus $3$. 

\begin{figure}
    \vspace{5mm}
    \centering
    \includegraphics[width=0.85\textwidth]{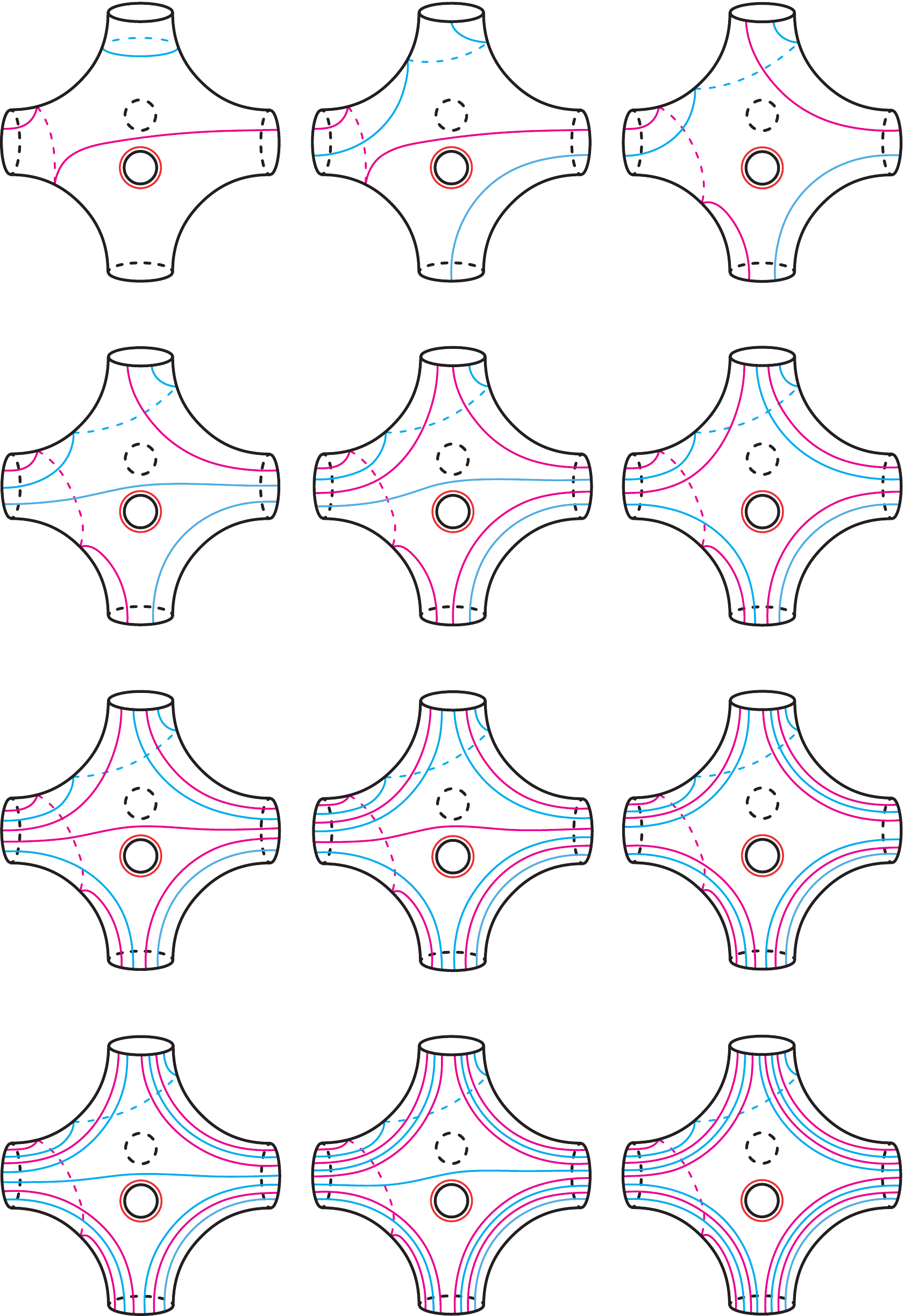}
    \put(-30,25){$C_{12}$}
    \put(-166,25){$C_{11}$}
    \put(-303,25){$C_{10}$}
    \put(-30,177){$C_9$}
    \put(-166,177){$C_8$}
    \put(-303,177){$C_7$}
    \put(-30,329){$C_6$}
    \put(-166,329){$C_5$}
    \put(-303,329){$C_4$}
    \put(-30,482){$C_3$}
    \put(-166,482){$C_2$}
    \put(-303,482){$C_1$}
    \put(-387,542){$d_2$}
    \vspace{3mm}
    \caption{The twelve cut systems describing the cobordism $A\subset E(1)$ from $T^3$ to $T^3$. The curves are colored differently only for visual clarity.}
    \label{fig:CobordismMultisection}
\end{figure}
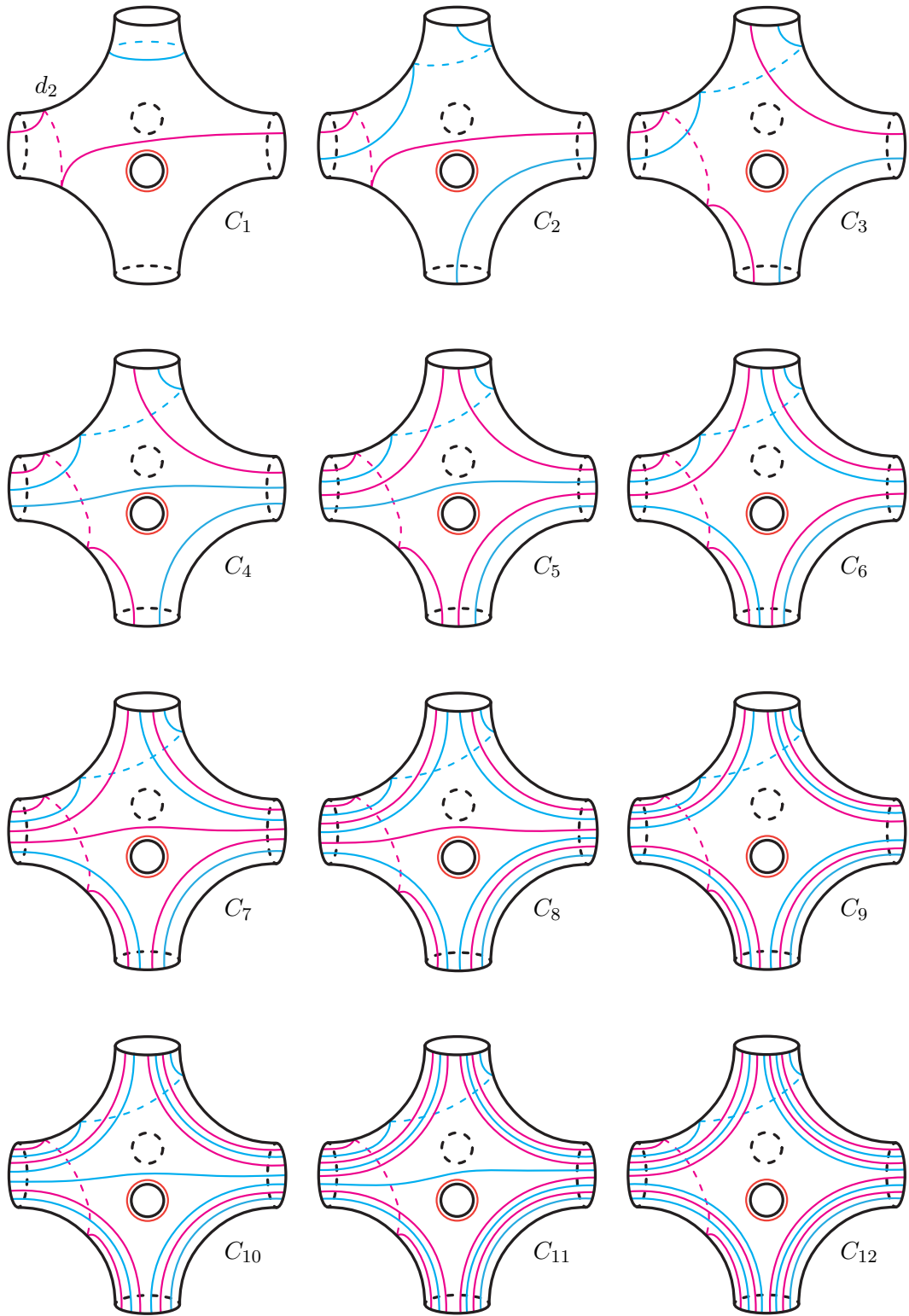

\subsection{Fiber Sums and Log Transforms}

Elliptic fibrations can be glued together along fibers in order to obtain new fibrations. More explicitly, let $p_i:X_i \to S^2$ $i\in \{1,2\}$ be smooth elliptic fibrations over $S^2$ and $F_i \subset X_i$ be regular fibers. We may form the \emph{fiber sum}, denoted $X_1 \#_f X_2$, by removing neighbourhoods of $F_i$, and gluing the resulting manifolds together by an orientation reversing map of $T^3$. By a theorem of Moishezen \cite{Moi77}, the resulting smooth manifold does not depend on the choice of gluing map. Note that $X_1 \#_f X_2$ inherits the structure of an elliptic fibration over the sphere. We define $E(n)$ to be the elliptic fibration obtained by taking the $n$-fold fiber sum of $E(1)$ with itself, i.e., $E(n) = E(n-1) \#_f E(1)$. 

We may obtain a multisection diagram for $E(2)$ by removing bisections of $T^2 \times D^2$ from two copies of $E(1)$ and gluing the resulting multisections together. To account for the change in orientation, we reflect $T^3$ about an arbitrary coordinate $T^2$. Note that such a reflection is in the Goeritz group of the genus 3 Heegaard splitting of $T^3$, and so gluing along this map is compatible with the multisection structure. If $C_i$ is a cut system on the Heegaard surface $\Sigma$ for $T^3$, we let $\overline{C}_i$ be the result of reflecting this cut system along the chosen $T^2$. Then a multisection diagram for $E(2)$ is given by the ordered set of cut systems $$(D_3, D_2, D_1, C_1, C_2, \dots, C_{11}, C_{12} = D_1 = \overline{D}_1,  \overline{C}_1, \overline{C}_2, \dots, \overline{C}_{12}, D_2).$$ The manifolds $E(n)$ can be constructed similarly, by alternating between the ordinary and reflected cut systems.

Next, we describe a surgery operation on elliptic fibrations which preserves their structure. Let $p:X \to C$ be a smooth elliptic fibration, and $F \subset X$ be a regular fiber. A neighbourhood of $F$ is diffeomorphic to $T^2 \times D^2$, and so the fibration structure on $X$ induces a map $\boundary \nu(F)\cong T^2\times S^1 \to S^1$. Let $\phi_p: T^2 \times S^1 \to T^2 \times S^1 $ be a diffeomorphism such that the restriction of the map to the $S^1$ factor is a degree $p$ map. Then, the manifold $X- (\nu(F)) \cup_{\phi_p} T^2 \times D^2$ is called a \emph{log transform of multiplicity $p$}. By a theorem of Gompf \cite{Gom91}, the diffeomorphism type of the surgered manifold in this setting depends only on $p$. We denote the manifold obtained by doing $n$ log transforms of multiplicity $p_1,\dots, p_n$ on $E(n)$ by $E(n)_{p_1, \dots, p_n}$. 

If $p$ and $q$ are relatively prime, then $E(n)_{p,q}$ is simply connected. Moreover, it is straightforward to show that the intersection form is not changed by this surgery, and so it follows from Freedman's theorem \cite{Fre82} that these manifolds are homeomorphic. On the other hand, these manifolds often fail to be diffeomorphic. By identifying log transforms with rational blowdowns, Fintushel and Stern \cite{FinSte97} were able to calculate the Seiberg-Witten invariants of $E(n)_{p,q}$, for $n \geq 2$. Combining this with a result of Friedman \cite{Fri95} on the manifolds $E(1)_{p,q}$ gives the following theorem.

\begin{theorem}[\cite{FinSte97},\cite{Fri95}]\label{thm:EnpqDiffeoType}
Let $n \geq 1$, and suppose $p,q,p',q'>1$. Then, $E(n)_{p,q}$ is diffeomorphic to $E(n)_{p',q'}$ if and only if $\{p,q\} = \{p',q'\}$ as unordered pairs.
\end{theorem}

In order to understand the effect of a multiplicity $p$ log transform on the multisection of $E(n)$, we must understand the effect of a map $\phi_p: T^3 \to T^3$ on the genus 3 Heegaard splitting of $T^3$. Boileau and Otal \cite{BoiOta90} show that there is only one genus 3 Heegaard splitting of $T^3$ up to isotopy, and so any choice of $\phi_p$ will fix this Heegaard splitting, up to isotopy. Thus, cutting and regluing $E(n)$ along the two subsections diffeomorphic to $T^2 \times D^2$ produces multisections of $E(n)_{p,q}$. 

In order to draw diagrams of these log transforms, we will need a concrete representative of $\phi_p$. A mapping class of $T^3$ is determined entirely by its action on $H_1(T^3) = \mathbb{Z} \oplus \mathbb{Z} \oplus \mathbb{Z}$, so $\text{Mod}(T^3)$ is isomorphic to $SL_3(\mathbb{Z})$. We will continue to use our parameterization of $T^3\cong S^1\times S^1\times S^1$, where $\{\text{pt}\}\times \{\text{pt}\}\times S^1$ bounds a disk in $T^2\times D^2$. For convenience, denote the loop corresponding to the $i$-th coordinate by $x_i$, so that $H_1(T^3)$ is generated by $\{[x_1],[x_2],[x_3]\}$. Recall that $SL_3(\mathbb{Z})$ is generated by the matrices $A_{i,j}$, where $A_{i,j}$ is the elementary matrix which only differs from the identity matrix by a $1$ in the $(i,j)$ position. For $i,j,k$ distinct, we may realize the mapping class corresponding to $A_{i,j}$ by twisting the $(x_i, x_k)$ coordinate torus in the $x_i$ direction. This map is illustrated in Figure \ref{fig:torusTwist}, and its effect on the Heegaard surface is shown in Figure \ref{fig:torusTwistonHeegaardSurface}.

\begin{figure}[ht]
    \centering
    \includegraphics[scale=.3]{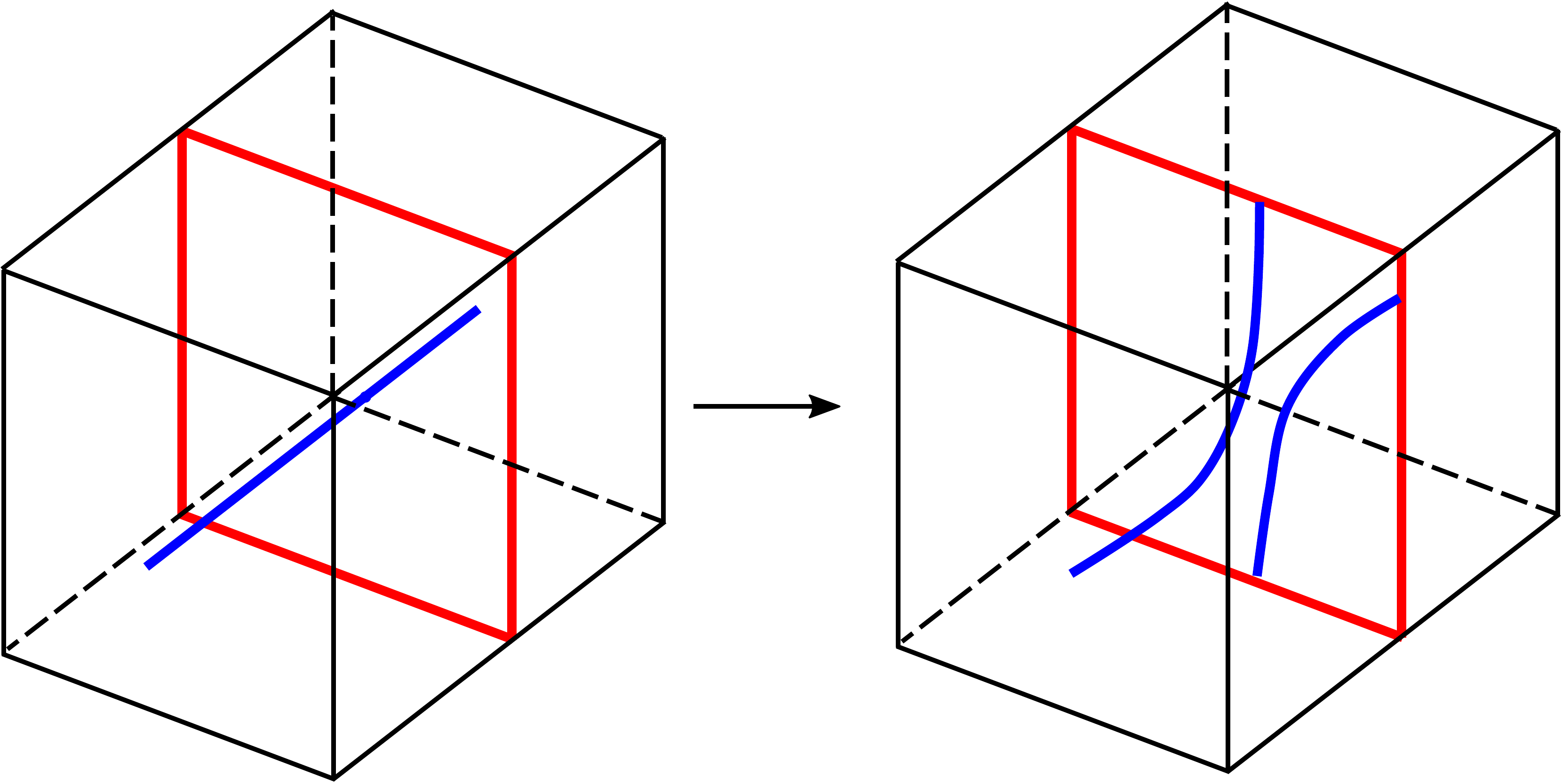}
    \caption{The effect on $x_i$ of a torus twist about the $x_j,x_k$ coordinate $T^2$ in $T^3$.}
    \label{fig:torusTwist}
\end{figure}

\begin{figure}[ht]
    \centering
    \includegraphics[scale=.3]{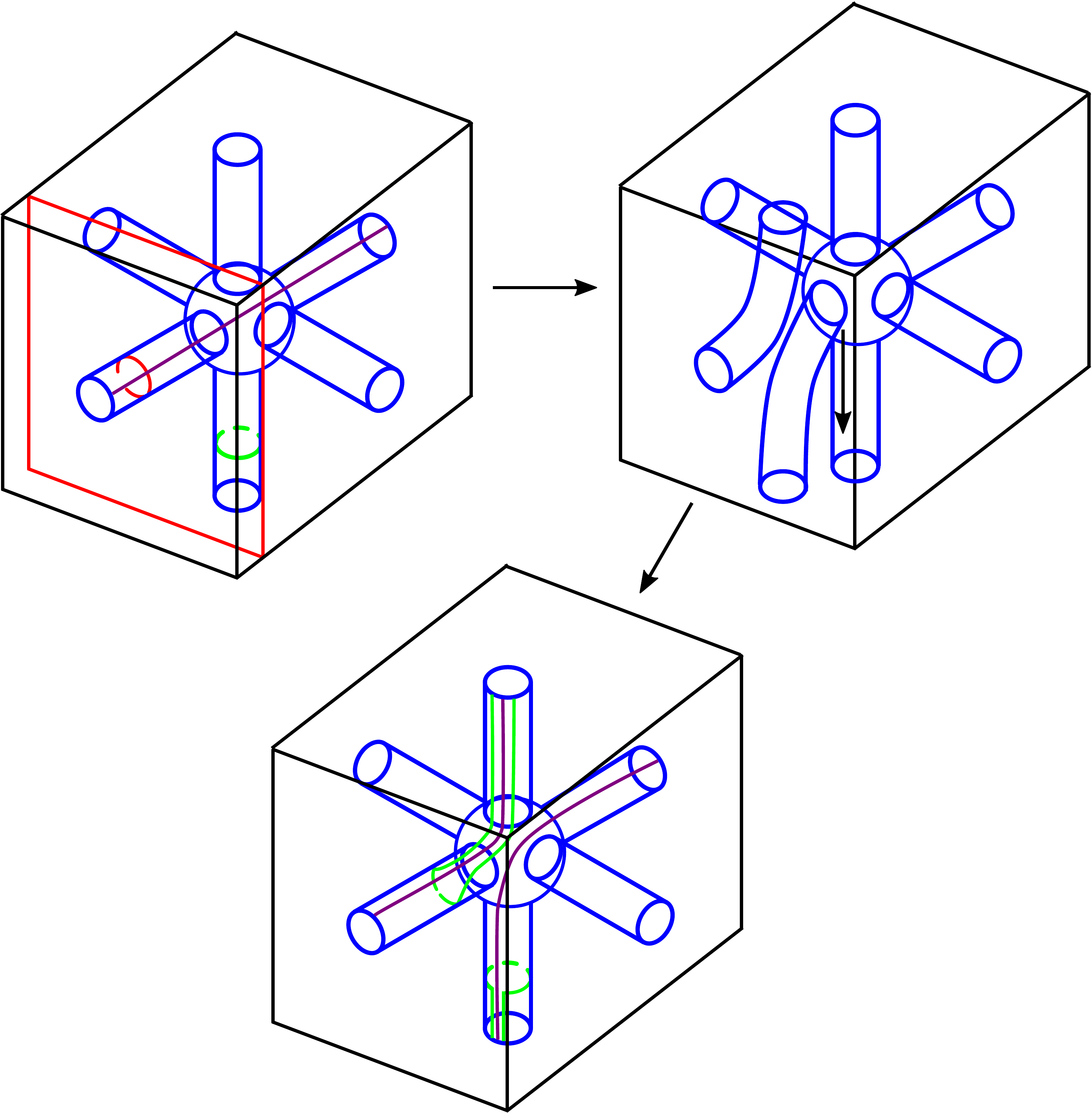}
    \caption{Top left: the genus 3 Heegaard splitting meets each coordinate torus in a circle. We will perform a torus twist in the upwards direction. Top Right: performing a torus twist about a coordinate torus moves the surface, but dragging the foot of the tube in the direction of the black arrow returns the surface to its original position. Bottom: the dragging process induces a point pushing map on the Heegaard surface, which  transforms the original green and purple curves as illustrated.}
    \label{fig:torusTwistonHeegaardSurface}
\end{figure}

\begin{figure}
    \centering
    \includegraphics[width=0.85\textwidth]{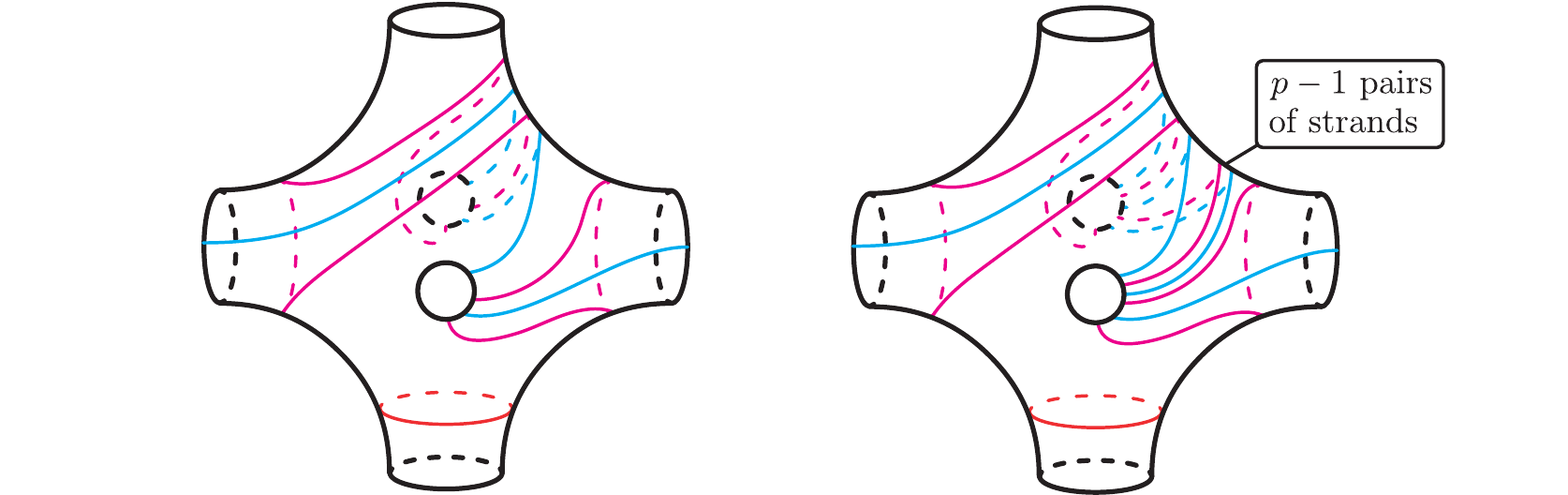}
    \put(-87,25){$D_2^p$}
    \put(-255,25){$D_2^2$}
    \caption{A log transform of multiplicity two or $p$ (pictured: $p=3$) transforms $C_\beta$ in Figure \ref{fig:T2xD2} to the cut system shown above.}
    \label{fig:logTransform}
\end{figure}

With respect to our chosen basis, we may take $\phi_p$ to be any matrix with a $p$ in the bottom right corner, and so we will choose the particular matrix
$$\phi_p=\begin{bmatrix} 1 & 0 & 0\\ 0 & 1 & 1 \\ 0 & p-1 & p \end{bmatrix}=A_{3,2}^{p-1}A_{2,3}=\begin{bmatrix} 1 & 0 & 0\\ 0 & 1 & 0 \\ 0 & 1 & 1 \end{bmatrix}^{p-1}\begin{bmatrix} 1 & 0 & 0\\ 0 & 1 & 1 \\ 0 & 0 & 1 \end{bmatrix}.$$

\noindent Therefore, this corresponds to a twist along the $(x_1, x_2)$ coordinate torus in the $x_1$ direction, followed by $p-1$ twists along the $(x_1, x_3)$ coordinate torus in the $x_3$ direction. 

The result of composing these maps fixes the handlebodies described by $D_1$ and $D_3$ in Figure \ref{fig:T2xD2}, but does not fix the handlebody described by $D_2$. The effect of $\phi_2$ on $D_2$ is shown on the left of Figure \ref{fig:logTransform}. More generally, the effect of $\phi_p$ on $D_2$ is described on the right of Figure \ref{fig:logTransform}, and we denote this cut system by $D_2^p$. Since $(D_1,D_2,D_3)$ appears twice in the multisection diagrams for $E(n)$, replacing $D_2$ in each triple with a $D_2^p$ and $D_2^q$ produces multisection diagrams for $E(n)_{p,q}$. Following Theorem \ref{thm:EnpqDiffeoType} we obtain the following result.

\infManyExoticGenusThree

\section{Multisections and stable maps}\label{sec:StableMaps}

\subsection{Stable maps of 4-manifolds to the disk}

In this section we discuss connections to stable maps of 4-manifolds to the disk. The reader is referred to \cite{BaySae17} for an overview of the topic well suited to our setting and to \cite{GolGui73} for the theoretical underpinnings of the subject. The smooth maps $X^4 \to D^2$ whose singular set consists of folds and cusps are stable in the category of smooth maps. In the process of proving Theorem 4 of \cite{GayKir16}, the authors show that the manifold $\natural^k S^1 \times B^3$ admits a map to a wedge whose singular image is shown on the left of Figure \ref{fig:wedgeMap}. There is a definite fold on the outside of the wedge, followed by $k$ indefinite folds (without cusps) and $g-k$ indefinite folds with a cusp. The fiber genus increases in the direction of the arrow, i.e., the folds are of index 1 when moving towards corner of the wedge. Note that such a map induces a genus $g$ Heegaard splitting of $\#^k S^1 \times S^2$.

Now, suppose that $X$ is a multisected 4-manifold with sectors $X_1,\dots,X_n$, and of genus $g$. Each sector induces such a map of $\natural^{k_i}S^1\times B^3$, and so induces two Morse functions on the intersections $H_{i,i+1}$. Any two Morse functions are homotopic, and so these maps may be connected by Cerf boxes, i.e., regions where critical levels may change to facilitate handle sliding, but no births or deaths occur. A typical example of such a map is given in the right side of Figure \ref{fig:wedgeMap}. Given a Morse 2-function, we can also explicitly extract a multisection (compare with \cite{GayKir16}).

\begin{figure}
    \centering
    \includegraphics[scale=.3]{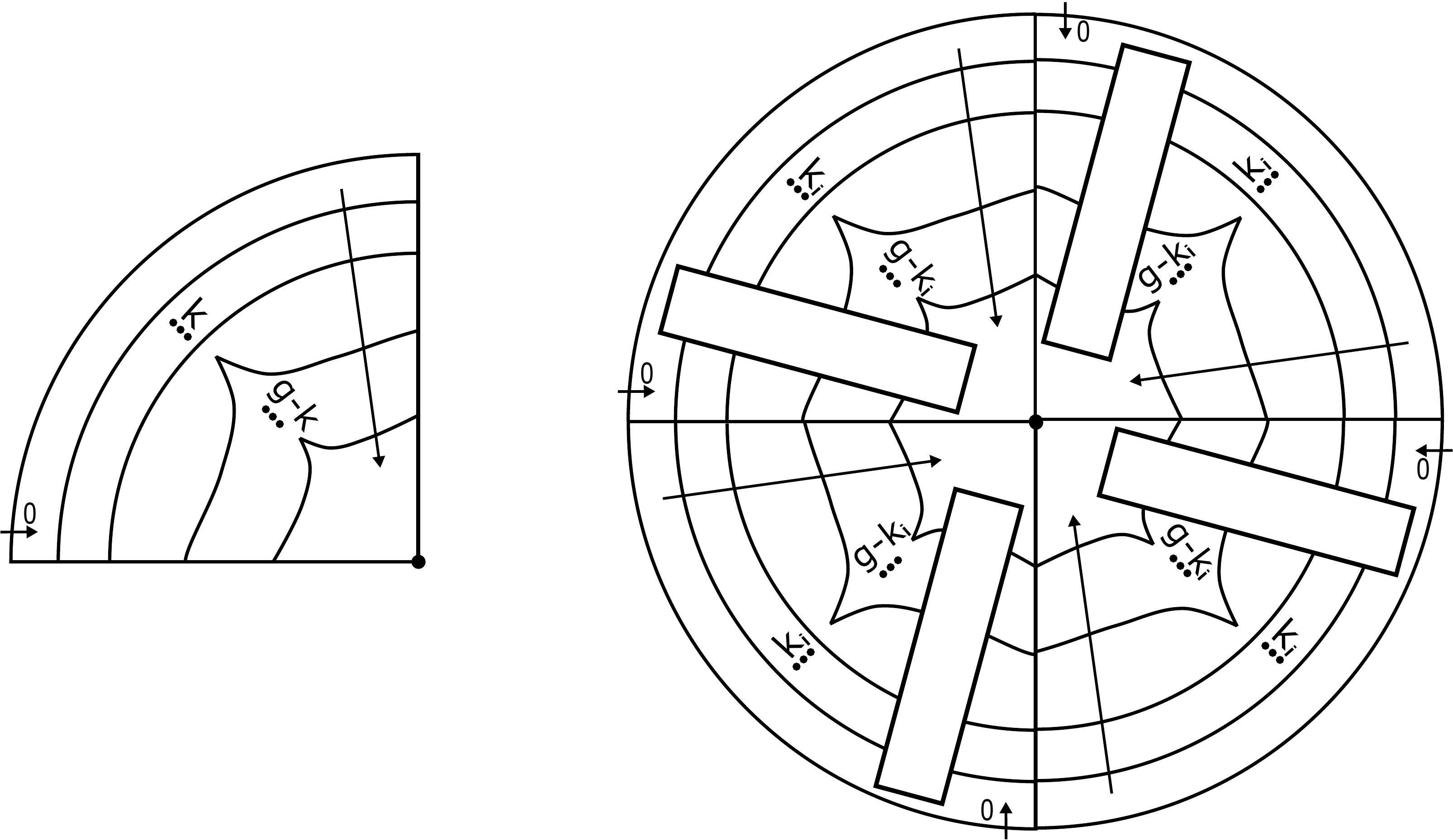}
    \caption{Left: the singular image of a map from $\natural^k S^1 \times B^3$ to a wedge. Right: a 4-section gives rise to four of these maps pasted together via Cerf boxes.}
    \label{fig:wedgeMap}
\end{figure}

\begin{definition}
A Morse 2-function $f:X^4 \to D^2$ is said to be \emph{radially monotonic} if
\begin{enumerate}
    \item $f$ has a unique definite fold;
    \item Indefinite folds always have index 1 when moving towards the center of the disk;
    \item There exist $n$ radial lines separating the disk into $n$ sectors, such that each fold has at most one cusp in each sector.
\end{enumerate}
\end{definition}

By the conditions above, the sectors of a radially monotonic function $f:X^4 \to D^2$ are diffeomorphic to $\natural^k S^1 \times B^3$. Two sectors sharing a radial line meet in a handlebody, and two sectors which do not share a radial line meet in the surface lying in the inverse image of the central point of the disk. Therefore, the inverse images of the sectors of a radially monotonic Morse 2-function are a multisection of $X$. The advantage of this singularity theoretic perspective is that we may modify the critical image, or its decomposition into sectors, in order to obtain different multisections of a fixed 4-manifold. In particular, we can perform expansion operations (see Definition \ref{def:expansion}) to produce a thin multisection.

\begin{proposition}\label{prop:thinMultisections}
Let $X$ be a multisected 4-manifold. Then, there is a sequence of expansions of this multisection producing a thin multisection.
\end{proposition}

\begin{proof}
Denote the sectors of this multisection by $X_1,\dots,X_n$, and let $f:X\to D^2$ be a radially monotonic Morse 2-function inducing this multisection. If any sector contains more than two cusps, draw an additional line between the cusps to separate the region into two sectors; this new multisection is related to the old one by a thinning operation. After separating each cusp in the decomposition, we are left with a collection of sectors, each of which contains a single cusp. Such a sector is diffeomorphic to $\natural^{g-1} S^1 \times B^3$, where $g$ is the genus of the multisection. By definition, this multisection is thin.
\end{proof}

Following \cite{BaySae17}, we say a local modification of a critical images which takes a critical image $C_0$ and produces a critical image $C_1$ is \emph{always realizable} if there is a smooth 1-parameter family of smooth maps $f_t: X^4 \to D^2$ such that the critical image of $f_0$ is $C_0$ and the critical image of $f_1$ is $C_1$. Three always realizable critical image modifications are shown in Figure \ref{fig:alwaysRealizeableMoves}. The first is the \emph{unsink} move, which takes a fold with index 1 and an index 2 critical point meeting at a cusp, and transforms it into a fold with no cusp together with a Lefschetz singularity. The second is the \emph{push} move, which moves a Lefschetz singularity over an indefinite fold in the direction of increasing fiber genus. The third is the \emph{wrinkle} move, which transforms a Lefschetz singularity into a fold with three cusps. The first and third moves were introduced by Lekeli in \cite{Lek09} and the second was introduced by Baykur in \cite{Bay09}. The reader is referred to these sources for the proofs that these moves are always realizable. For convenience, we will define a composition of these moves for radially monotonic Morse 2-functions. 

\begin{figure}
    \centering
    \includegraphics[scale=1]{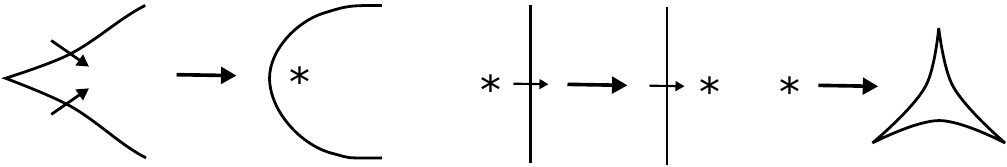}
    \caption{A collection of always realizable moves. Left: the unsink move replaces a cusp with a Lefschetz singularity. Middle: one can always push a Lefschetz singularity to the higher genus side of an indefinite fold. Right: a Lefschetz singularity can be wrinkled into a fold with three cusps.}
    \label{fig:alwaysRealizeableMoves}
\end{figure}

\begin{definition}
Let $f: X^4 \to D^2$ be a radially monotonic Morse 2-function and $c$ a cusp in $\textnormal{crit}(f)$. A \emph{UPW-move} (unsink-push-wrinkle) is the modification of $\textnormal{crit}(f)$ given by unsinking $c$, pushing the resulting Lefschetz singularity to the center of the disk, and wrinkling the singularity.
\end{definition}

Since the unsink, push, and wrinkle moves are always realizable, so is the UPW move. This move preserves the radial monotonicity of a stable function and hence takes a genus $g$ multisection to a genus $g+1$ multisection. Note that when $n>3$, we must choose how to distribute the three new cusps among the sectors. This does not change the manifold, but we may need to take extra care when we are interested in the resulting multisection.

\subsection{Decreasing the number of sectors and stable equivalence}

In this subsection, we will use the moves described above to relate any two multisections of a fixed 4-manifold. These moves can be used to decrease the number of sectors in a multisection, and so convert any multisection into a trisection. This is essential for our proof, since we will first convert a multisection into a trisection, and then use the stable equivalence of trisections \cite{GayKir16} to deduce the stable equivalence of multisections. We will first show how to decrease the number of sectors in a multisection; the reader may consult Figure \ref{fig:UPWonQuad} for an illustration of the argument.

\begin{proposition}\label{prop:turnIntoTri}
Let $n>3$, and suppose a 4-manifold $X$ admits an $n$-section of genus $g$ with sectors $X_1,\dots,X_n$. Then, $X$ admits an $(n-1)$-section of genus $(2g-k_1)$, where $X_1\cong \natural^{k_1}S^1\times B_3$.
\end{proposition}

\begin{proof}
Let $f:X\to D^2$ be a Morse 2-function inducing the given multisection, so that $X_1$ is realized as the inverse image of a sector of $D^2$. This sector contains $g$ folds, and $g-k_1$ of them have cusps. Perform UPW moves on each of these cusps, distributing the three new cusps to $X_2, X_3,$ and $X_4$ each time (since $n>3$, there are at least three available sectors). Since we have done $g-k_1$ UPW moves, the resulting multisection has genus $g+(g-k_1) = 2g - k_1$ as desired. Furthermore, the resulting critical image set in $X_1$ now has no cusps. By Lemma \ref{lem:removeRedundantSector} this sector can be contracted into an adjacent sector, turning the original $n$-section into an $(n-1)$-section.
\end{proof}

\begin{figure}
    \centering
    \includegraphics[scale=.16]{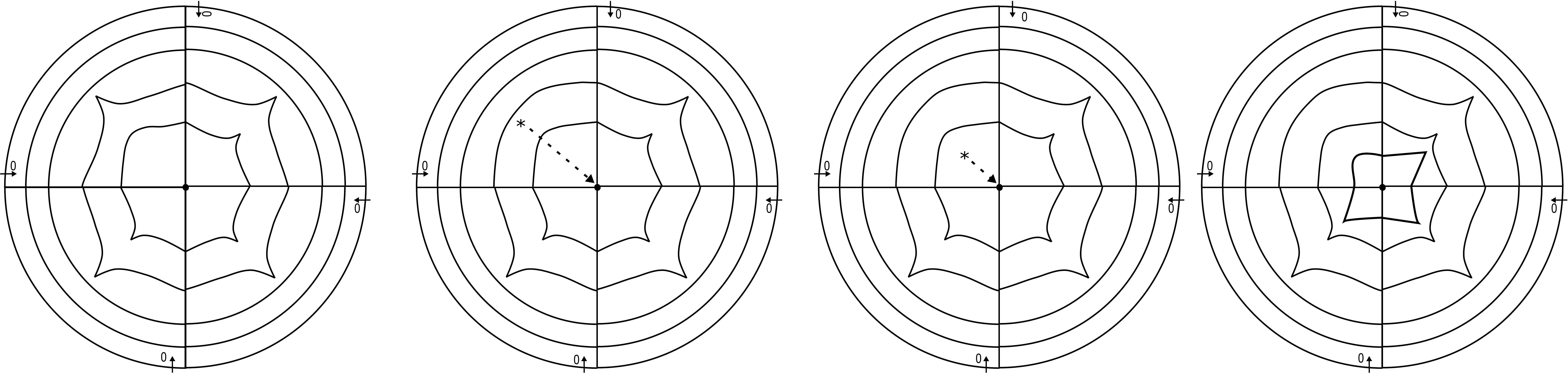}
    \caption{The sequence of moves making up a UPW move. Here, it is used to remove the cusp in the top left sector, leaving this sector with no cusps in the final critical image. A sector with no cusps can be contracted into either adjacent sector, producing a multisection with fewer sectors (in this case, a trisection).}
    \label{fig:UPWonQuad}
\end{figure}

In what follows, we will use the UPW move, as well as the stabilization move in Definition \ref{def:Stabilization}. We can a diagrammatic definition, but it also may be realized as a modification of a stable map. In \cite{GayKir16} (see Figure 13), Gay and Kirby realize this stabilization operation as the introduction of an ``eye,'' corresponding to a birth and death of a cancelling pair of indefinite critical points. We conclude with a notion of equivalence between any two multisections of a given 4-manifold.

\begin{theorem}
Let $X$ be a smooth, oriented, closed, and connected 4-manifold. Any two multisections of $X$ are related by a sequence of UPW moves, stabilizations, and isotopy through multisections.
\end{theorem}

\begin{proof}
Begin with two arbitrary multisections of $X$, with sectors $X_1,\dots,X_n$ and $Y_1,\dots,Y_m$. If either multisection has more than three sectors, we may use Proposition \ref{prop:turnIntoTri} to decrease the number of sectors until each is a trisection. By Theorem 11 of \cite{GayKir16}, these trisections are related by a sequence of stabilizations and isotopies.
\end{proof}

\bibliographystyle{alpha}

\end{document}